\documentclass[12pt,dvips]{amsart}
\usepackage{amsfonts, amssymb, latexsym, epsfig}

\newtheorem{Theorem}{Theorem}[section]
\newtheorem{Proposition}[Theorem]{Proposition}
\newtheorem{Lemma}[Theorem]{Lemma}
\newtheorem{Claim}[Theorem]{Claim}
\newtheorem{Subclaim}[Theorem]{Subclaim}

\newtheorem{Definition-Proposition}[Theorem]{Definition-Theorem}
\newtheorem{Main Conjecture}[Theorem]{Main Conjecture}

\newtheorem{Conjecture}[Theorem]{Conjecture}

\theoremstyle{remark}
\newtheorem{Example}[Theorem]{Example}
\newtheorem{Remark}[Theorem]{Remark}

\newcommand\x{{\sf x}}
\newcommand\y{{\sf y}}
\newcommand\z{{\sf z}}
\newcommand\bbb{{\sf b}}
\theoremstyle{plain}

\newcommand{\excise}[1]{}%{$\star$\textsc{#1}$\star$}

\newcommand{\cellsize}{19}
\newlength{\cellsz} 
\newsavebox{\cell}
\sbox{\cell}{\begin{picture}(\cellsize,\cellsize)
\put(0,0){\line(1,0){\cellsize}}
\put(0,0){\line(0,1){\cellsize}}
\put(\cellsize,0){\line(0,1){\cellsize}}
\put(0,\cellsize){\line(1,0){\cellsize}}
\end{picture}}
\newcommand\cellify[1]{\def\thearg{#1}\def\nothing{}%
\ifx\thearg\nothing
\vrule width0pt height\cellsz depth0pt\else
\hbox to 0pt{\usebox{\cell} \hss}\fi%
\vbox to \cellsz{
\vss
\hbox to \cellsz{\hss$#1$\hss}
\vss}}
\newcommand\tableau[1]{\vtop{\let\\\cr
\baselineskip -16000pt \lineskiplimit 16000pt \lineskip 0pt
\ialign{&\cellify{##}\cr#1\crcr}}}
\newcommand{\kellsize}{12}
\newlength{\kellsz} 
\newsavebox{\kell}
\sbox{\kell}{\begin{picture}(\kellsize,\kellsize)
\put(0,0){\line(1,0){\kellsize}}
\put(0,0){\line(0,1){\kellsize}}
\put(\kellsize,0){\line(0,1){\kellsize}}
\put(0,\kellsize){\line(1,0){\kellsize}}
\end{picture}}
\newcommand\kellify[1]{\def\thearg{#1}\def\nothing{}%
\ifx\thearg\nothing
\vrule width0pt height\kellsz depth0pt\else
\hbox to 0pt{\usebox{\kell} \hss}\fi%
\vbox to \kellsz{
\vss
\hbox to \kellsz{\hss$#1$\hss}
\vss}}
\newcommand\ktableau[1]{\vtop{\let\\\cr
\baselineskip -16000pt \lineskiplimit 16000pt \lineskip 0pt
\ialign{&\kellify{##}\cr#1\crcr}}}

\newcommand{\sellsize}{63}
\newlength{\sellsz} 
\newsavebox{\sell}
\sbox{\sell}{\begin{picture}(\sellsize,20)
\put(0,0){\line(1,0){\sellsize}}
\put(0,0){\line(0,1){\sellsize}}
\put(\sellsize,0){\line(0,1){\sellsize}}
\put(0,\sellsize){\line(1,0){\sellsize}}
\end{picture}}
\newcommand\sellify[1]{\def\thearg{#1}\def\nothing{}%
\ifx\thearg\nothing
\vrule width0pt height\sellsz depth0pt\else
\hbox to 0pt{\usebox{\sell} \hss}\fi%
\vbox to \sellsz{
\vss
\hbox to \sellsz{\hss$#1$\hss}
\vss}}
\newcommand\stableau[1]{\vtop{\let\\\cr
\baselineskip -16000pt \lineskiplimit 16000pt \lineskip 0pt
\ialign{&\sellify{##}\cr#1\crcr}}}
\begin{document}
\pagestyle{plain}

\mbox{}
\title{Equivariant Schubert calculus and jeu de taquin}
\author{Hugh Thomas}
\address{Department of Mathematics and Statistics, University of New Brunswick, Fredericton, New Brunswick, E3B 5A3, Canada }
\email{hugh@math.unb.ca}

\author{Alexander Yong}
\address{Department of Mathematics, University of Illinois at
Urbana-Champaign, Urbana, IL 61801, USA}

\email{ayong@uiuc.edu}

\date{July 13, 2012}

\maketitle
\begin{abstract}
We introduce edge labeled Young tableaux. Our main results provide a corresponding analogue of
[Sch\"{u}tzenberger '77]'s theory of \emph{jeu de taquin}. These are applied to the
equivariant Schubert calculus of Grassmannians. Reinterpreting, we present new (semi)standard tableaux
to study factorial Schur polynomials, after [Biedenharn-Louck '89],
[Macdonald '92] and [Goulden-Greene '94] and others.

Consequently,
we obtain new combinatorial rules for the Schubert structure coefficients,
complementing
work of [Molev-Sagan '99], [Knutson-Tao '03], [Molev '08] and [Kreiman '09]. We also
describe a conjectural generalization of one of our rules to the equivariant $K$-theory
of Grassmannians, extending work of [Thomas-Yong '07]. This conjecture concretely realizes the
``positivity'' known to exist by [Anderson-Griffeth-Miller '08].
It provides an alternative to the conjectural rule of Knutson-Vakil reported in
[Coskun-Vakil '06].
\end{abstract}

%\tableofcontents
%\newpage

\section{Introduction}

\subsection{Overview} The main goal of this paper is to introduce
edge labeled Young tableaux, together with
a corresponding analogue of the theory of \emph{jeu de taquin}. We apply them to the
setting of equivariant Schubert calculus of Grassmannians.
This paper may also be interpreted as extending
(semi)standard tableaux for use with the closely related family of
factorial Schur polynomials. %By comparison, previous work \cite{Louck, Goulden, Macdonald, Molev, Kreiman:Eq}
%had ascribed roles to semistandard tableaux different than ours.

The classical theory of \emph{jeu de taquin}, initiated by M.-P.~Sch\"{u}tzenberger \cite{Schutzenberger}, has been of
significance in combinatorial representation theory. One outcome of this theory is a combinatorial rule for the Littlewood-Richardson
coefficients. Perhaps more importantly, it provides a systematic and flexible means to elegantly reconcile a
variety of important tableau algorithms. It achieves this using a simple sliding law.

The Littlewood-Richardson coefficients compute Schubert calculus of Grassmannians. More precisely, they are structure
coefficients for multiplication with respect to the Schubert basis of the ordinary cohomology ring of Grassmannians. Since a Grassmannian admits the action of the torus $T$ of invertible diagonal matrices, one can instead study the richer $T$-equivariant cohomology ring and its Schubert calculus.
While Littlewood-Richardson rules were already available for this setting \cite{Knutson.Tao}, further ideas
are needed to (provably) extend them to other Lie types or  finer cohomology theories. In addition,
to date, \emph{jeu de taquin} is the only combinatorial model that admits a root-system uniform rule for Schubert calculus on minuscule $G/P$'s  \cite{Thomas.Yong:II}. These are our
principal reasons for seeking new combinatorial models that extend \emph{jeu de taquin}.

\subsection{Schubert calculus of Grassmannians}

Let $X=Gr(k,{\mathbb C}^n)$ denote the Grassmannian of $k$-dimensional planes
in ${\mathbb C}^n$. If $\lambda=(n-k\geq \lambda_1\geq \lambda_2\geq\cdots\geq \lambda_k\geq 0)$ is
a Young diagram contained in the rectangle $\Lambda:=k\times (n-k)$, the associated
{\bf Schubert variety} is defined by
\[X_\lambda:=\left\{V\in Gr(k,{\mathbb C}^n)|\dim(V\cap F^{n-k+i-\lambda_i})\geq i,
1\leq i\leq k\right\},\]
where $F^d={\rm span}(e_{n},e_{n-1},\ldots,e_{n-d+1})$. With this convention,
${\rm codim}(X_{\lambda})=|\lambda|=\sum_{i}\lambda_i$.

Let
$T\subseteq GL_{n}$ be the torus of invertible diagonal matrices.
Since $X_{\lambda}$ is $T$-stable under the action of $T$ on $X$,
$X_{\lambda}$ admits an
{\bf equivariant Schubert class} $\sigma_{\lambda}$ in $H_T(X)=$ the $T$-equivariant cohomology
ring of $X$. Now, $H_T(X)$ is a module over
$H_{T}({\rm pt}):={\mathbb Z}[t_1,\ldots,t_n]$,
and these classes form an
additive $H_T({\rm pt})$-basis of $H_T(X)$. The expansion
\begin{equation}
\label{eqn:eq_expansion}
\sigma_{\lambda}\cdot \sigma_{\mu}=\sum_{\nu}
C_{\lambda,\mu}^{\nu}\sigma_{\nu},
\end{equation}
defines the {\bf equivariant Schubert structure coefficients}
$C_{\lambda,\mu}^{\nu}\in {\mathbb Z}[t_1,\ldots,t_n]$.
In fact, $C_{\lambda,\mu}^{\nu}=0$ unless $|\lambda|+|\mu|\geq |\nu|$. In the case
of equality, $C_{\lambda,\mu}^{\nu}\in {\mathbb N}$ are the {\bf Littlewood-Richardson coefficients}; these compute the number of points of
$X$ in $g_1\cdot X_{\lambda}\cap g_2\cdot X_{\mu}\cap
g_3\cap X_{\nu^{\vee}}$,
where $g_1,g_2,g_3$ are generic elements of $GL_n$ and $\nu^{\vee}$ is the
180-degree rotation of the complement of $\nu$ inside $\Lambda$.

W.~Graham \cite{Graham} proved that the polynomials
$C_{\lambda,\mu}^{\nu}$ have positive coefficients when expressed in the
variables $\beta_i:=t_{i}-t_{i+1}$. This positivity is evident in the statement of
A.~Knutson-T.~Tao's combinatorial \emph{puzzle} rule \cite{Knutson.Tao}.
Later, alternative tableau rules were given by
V.~Kreiman \cite{Kreiman:Eq} and A.~Molev \cite{Molev} (in these rules, the positivity is not
hard to prove). See also the work of P.~Zinn-Justin \cite{Zinn-Justin}.

\subsection{(Semi)standard tableaux with edge labels}

\label{section:eqsyt}
Our work depends on a new kind of Young tableaux.
Let ${\mathbb Y}$ denote the set of Young diagrams (drawn in English notation).
Given $\lambda,\nu\in {\mathbb Y}$ with $\lambda$ contained in $\nu$,
denote the skew shape by $\nu/\lambda$. A {\bf horizontal edge} of $\nu/\lambda$ is a horizontally-oriented line segment which either lies along the upper or
lower boundary of $\nu/\lambda$, or which separates two boxes of $\nu/\lambda$.

An {\bf equivariant filling} of $\nu/\lambda$ assigns one of the labels $1,2,\ldots,\ell$
to each box of $\nu/\lambda$ and a (possibly empty) subset of
$\{1,2,\ldots,\ell\}$ to each horizontal edge of $\nu/\lambda$.
An equivariant filling is {\bf semistandard} if every box label is:
\begin{itemize}
\item weakly smaller than the label in the box immediately to its right;
\item strictly smaller than any label in its southern edge and the label in the
box immediately below it; and
\item strictly larger than any label in its northern edge and the label in the box
immediately above it.
\end{itemize}
(No condition is placed on the labels of adjacent edges.) The filling is {\bf standard} if the labels used are $1, 2, \dots, \ell$, and each label is used
exactly once.

Let ${\tt EqSYT}(\nu/\lambda,\ell)$ and ${\tt EqSSYT}(\nu/\lambda,\ell)$
respectively be the set of equivariant standard and semistandard tableaux whose entries come from
$\{1,2,\ldots,\ell\}$. For example:
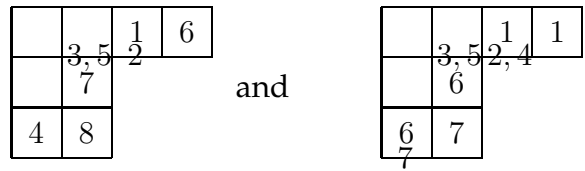
\[
\begin{picture}(300,60)
\put(20,40){$\tableau{{\ }&{\ }&{1}&{6 }\\{\ }&{7 }\\{4 }&{8 }}$}
\put(41,37){$3,5$}
\put(64,37){$2$}
\put(105,25){and}
\put(160,40){$\tableau{{\ }&{\ }&{1}&{1 }\\{\ }&{6 }\\{6 }&{7 }}$}
\put(181,37){$3,5$}
\put(200,37){$2,4$}
\put(166,-2){$7$}
\end{picture}
\]
which are in ${\tt EqSYT}((4,2,2)/(2,1),8)$ and ${\tt EqSSYT}((4,2,2)/(2,1),8)$, respectively.

Those $T\in {\tt EqSYT}(\nu/\lambda,|\nu/\lambda|)$, where each
horizontal edge has no labels, are in obvious
bijection with {\bf (ordinary) standard Young tableaux}. In this
latter case, we also call $T$ an ordinary standard tableau. (We will
drop ``equivariant'' for fillings unless confusion might arise.)
Finally, an ordinary standard Young tableau of shape $\mu$ is
{\bf row superstandard} if it is filled by $1,2,\ldots,\mu_1$ in the first
row, $\mu_1+1,\mu_1+2,\ldots,\mu_1+\mu_2$ in the second row, etc. Let
$T_{\mu}$ denote the row superstandard Young tableau of shape $\mu$.

\subsection{Equivariant jeu de taquin (first version)}
Our first version
of equivariant \emph{jeu de taquin} omits some features (and complexity) of the main construction
of Section~2. Nevertheless, this version already suffices to compute the polynomials
$C_{\lambda,\mu}^{\nu}$. Moreover, it suggests generalizations. Specifically,
we present a conjectural generalization to equivariant $K$-theory in Section~4. It also suggests a first step towards an extension to minuscule $G/P$'s (further discussion may appear elsewhere), cf.~\cite{Thomas.Yong:II}.

A box ${\sf x}\in\lambda$ is an {\bf inner corner} of $\nu/\lambda$
if it is maximally southeast in $\lambda$.
Given an inner corner $\x$ and $T\in {\tt EqSYT}(\nu/\lambda,\ell)$,
compare the label
in the box immediately to the right of $\x$ and the smallest
label on the southern edge of $\x$, or the label in the box immediately below
$\x$, if no label appears on that edge. The smaller of the labels
is moved into $\x$, either by vacating a box or
moving a label from the southern edge of $\x$. If no labels can be used or if
an edge label is moved, the process terminates. Otherwise, some
adjacent box has been vacated, and we repeat the above process until
termination. Call the result ${\tt Ejdt}_{\x}(T)$, the
{\bf equivariant jeu de taquin slide} into $\x$. Clearly,
${\tt Ejdt}_{\x}(T)$ is also a standard tableau.

Define \emph{the} equivariant rectification of $T$, denoted ${\tt Erect}(T)$,
to be the result of applying  the sequence
${\tt Ejdt}_{\x(1)}, {\tt Ejdt}_{\x(2)},\ldots, {\tt Ejdt}_{\x(|\lambda|)}$
starting with $T$, where $\x(1),$ $\x(2),\ldots, \x(|\lambda|)$ are the boxes
of $\lambda$, read along columns, from bottom to top, and
right to left.
\excise{By analogy with the classical jeu de taquin setting,
it would be fair to call this the {\bf column rectification order}. We will
make reference to the intermediate {\bf rectification of a column} $c$,
which concerns those ${\tt Eqjdt}_{x(i)}$ where $x(i)$ lies in column $c$.}

\begin{Example}
\label{exa:1.1}
Let $\nu/\lambda=(4,3,1)/(3,1,1)\subseteq \Lambda=3\times 4$ and
\[\begin{picture}(100,50)
\put(0,35){$T=\tableau{{\ }&{\ }&{ \ }&{3}\\{\ }&{5}&{6}\\{\ }}$}
\put(32,-7){$4$}
\put(49,31){$1$}
\put(69,31){$2$}
\end{picture}\]
We use ``$\bullet$'' to indicate the boxes being slid into during the steps of
${\tt Erect}(T)$. The rectification of the third column given by:
\begin{equation}
\label{eqn:ex1.1.1}
\begin{picture}(150,50)
\put(0,35){$\tableau{{\ }&{\ }&{ \bullet }&{3}\\{\ }&{5}&{6}\\{\ }}$}
\put(110,35){$\tableau{{\ }&{\ }&{2}&{3}\\{\ }&{5}&{6}\\{\ }}$}
\put(8,-7){$4$}
\put(25,31){$1$}
\put(45,31){$2$}
\put(118,-7){$4$}
\put(135,31){$1$}
\put(90,35){$\mapsto$}
\end{picture}
\end{equation}
The rectification of the second column given by:
\begin{equation}
\label{eqn:ex1.1.2}
\begin{picture}(150,50)
\put(0,35){$\tableau{{\ }&{\bullet }&{ 2}&{3}\\{\ }&{5}&{6}\\{\ }}$}
\put(110,35){$\tableau{{\ }&{1 }&{2}&{3}\\{\ }&{5}&{6}\\{\ }}$}
\put(8,-7){$4$}
\put(25,31){$1$}
\put(118,-7){$4$}
\put(90,35){$\mapsto$}
\end{picture}
\end{equation}
and finally the rectification of the first column given by:
\begin{equation}
\label{eqn:ex1.1.3}
\begin{picture}(400,50)
\put(0,35){$\tableau{{\ }&{1 }&{ 2}&{3}\\{\ }&{5}&{6}\\{\bullet }}$}
\put(110,35){$\tableau{{\ }&{1 }&{2}&{3}\\{\bullet }&{5}&{6}\\{4 }}$}
\put(220,35){$\tableau{{\bullet }&{1 }&{2}&{3}\\{4 }&{5}&{6}\\{\ }}$}
\put(350,35){$\tableau{{1 }&{2 }&{3}&{\bullet }\\{4 }&{5}&{6}\\{\ }}$}
\put(6,-7){$4$}
\put(90,35){$\mapsto$}
\put(195,35){$\mapsto$}
\put(301,35){$\mapsto\cdots\mapsto$}
\end{picture}
\end{equation}
the last tableau being $T_{(3,3)}$. Here the ``$\mapsto\cdots\mapsto$'' refers to slides moving the
$\bullet$ right in the first row.\qed
\end{Example}

We now define the {\bf weight}
${\tt wt}(T)\in {\mathbb Z}[t_1,\ldots,t_n]$
of a standard tableau $T$.
Each
box $\x\in\Lambda$ is assigned
a weight $\beta(\x)=t_m-t_{m+1}$ where $m$ is the ``Manhattan distance''
from the southwest corner (point) of $\Lambda$ to the northwest corner (point)
of $\x$ (i.e., the length of any north and east lattice path between the corners);
see Example~\ref{exa:Manh}. We say an edge label ${\mathfrak l}$ {\bf passes} through a box
$\x$ if it occupies $\x$ during the equivariant rectification of the column
of $T$ in which
${\mathfrak l}$ begins. Suppose that the boxes passed are
$\x_1,\x_2,\ldots,\x_s$. Moreover, once the rectification of a column is
complete, suppose the filled boxes strictly to the right of the box $\x_s$ are
$\y_1,\ldots,\y_t$. Then set
\[{\tt factor}({\mathfrak l})=(\beta(\x_1)+\beta(\x_2)+\cdots+\beta(\x_s))+
(\beta(\y_1)+\beta(\y_2)+\cdots+\beta(\y_t)).\]
If after rectification of a column, the label ${\mathfrak l}$
still remains an edge label, ${\tt factor}({\mathfrak l})$ is declared to be zero.
Otherwise, note that since the boxes $\x_1,\ldots,\x_s, \y_1,\ldots, \y_t$
form a hook inside $\nu$, ${\tt factor}(i)=t_e-t_f$ with $e<f$. Now define
\[{\tt wt}(T):=\prod_{{\mathfrak l}} {\tt factor}({\mathfrak l}),\]
where the product is over all edge labels ${\mathfrak l}$ of $T$.

\begin{Theorem}
\label{thm:main}
The equivariant Schubert structure coefficient is given by the polynomial
\[C_{\lambda,\mu}^{\nu}=\sum_T {\tt wt}(T),\]
where the sum is over all $T\in {\tt EqSYT}(\nu/\lambda,|\mu|)$ such that
${\tt Erect}(T)=T_{\mu}$.
\end{Theorem}

Since each ${\tt factor}({\mathfrak l})$ is a positive sum of the
indeterminates $\beta_i=t_{i}-t_{i+1}$, Theorem~\ref{thm:main}
expresses $C_{\lambda,\mu}^{\nu}$ as a polynomial with positive
coefficients in the $\beta_i$'s.
%, whereas the rules \cite{Kreiman:Eq, Molev} give
%\emph{provably}, but not manifestly,
%Graham positive expressions for $C_{\lambda,\mu}^{\nu}$.
%%% HT says: I think we've already made this point sufficiently.
It is not hard to see that Theorem~\ref{thm:main} expresses $C_{\lambda,\mu}^{\nu}$ as
a squarefree polynomial in the ``positive root'' variables $\alpha_{ij}=\beta_i+\beta_{i+1}+\cdots +\beta_{j-1}$, also a feature of the puzzle rule of \cite{Knutson.Tao}.

\begin{Example}
\label{exa:Manh}
Continuing Example~\ref{exa:1.1},
the Manhattan distances for $\Lambda=3\times 5$ are:
\[\ktableau{{3}&{4}&{5}&{6}&{7}\\{2}&{3}&{4}&{5}&{6}\\{1}&{2}&{3}&{4}&{5}}.\]
There are three edge labels of $T$:
\begin{itemize}
\item For the edge label $2$, we have
${\tt factor}(2)=(t_5-t_6)+(t_6-t_7)=t_5-t_7$ since the edge label passes
through one box, and after the third column is rectified (\ref{eqn:ex1.1.1}),
the $\ktableau{3}$
lies to its right.
\item For the edge label $1$, we have
${\tt factor}(1)=(t_4-t_5)+(t_5-t_6)+(t_6-t_7)=t_4-t_7$ since the edge label passes
through one box, and after the second column is rectified (\ref{eqn:ex1.1.2}),
the $\ktableau{{2}&{3}}$
lies to its right.
\item For the edge label $4$, we have
${\tt factor}(4)=(t_1-t_2)+(t_2-t_3)+(t_3-t_4)+(t_4-t_5)=t_1-t_5$ since the edge label passes
through two boxes, and after the first column is rectified (\ref{eqn:ex1.1.3}),
$\ktableau{{5}&{6}}$ lies to its right.
\end{itemize}
Therefore, ${\tt wt}(T)=(t_5-t_7)(t_4-t_7)(t_1-t_5)$.\qed
\end{Example}

In Sch\"{u}tzenberger's \emph{jeu de taquin} theory, one is free to slide at
different inner corners. His theory's ``first fundamental theorem'' is that rectification does not depend on these choices.
The above equivariant \emph{jeu de taquin} avoids this issue altogether by insisting on a specific
order of rectification.
Even more, the classical theory's ``second fundamental theorem'' asserts the number of tableaux that rectify to a given target tableau is independent of the
choice of target tableau. In contrast, we insist on using row superstandard
tableaux as our targets.

The above rigid definition of \emph{jeu de taquin} makes nonobvious to us how to directly prove
Theorem~\ref{thm:main}. Although one can biject the rule of Theorem~\ref{thm:main}
with earlier rules, our original reason for starting this project was to find a model that could ultimately extend to other equivariant contexts where earlier rules are unavailable.
\excise{Specifically, we have in mind that situation where $G/P$ is (co)minuscule, or, as treated conjecturally in this paper when we work in equivariant $K$-theory.}

Therefore, our problem was to find a more flexible
version of equivariant \emph{jeu de taquin} possessing features of the fundamental theorems.
Our solution is described in Section~2. It has some aspects that are distinctly different
than the classical \emph{jeu de taquin} (and our first version of equivariant \emph{jeu de taquin}):
\begin{itemize}
\item More than one label can move during a swap.
\item Labels can move downwards during a swap.
\item Row semistandardness can be violated after a swap (although at
most one such violation occurs at any given time, and it is
eliminated at the end of a sequence of swaps that defines a slide).
\end{itemize}

 Our main result shows that the order of rectification is independent of the choices, if one  rectifies to a ``highest weight tableau'' and starts with a
tableau that is ``lattice''. From this, we derive an essentially independent proof of Theorem~\ref{thm:main}.

\subsection{Organization}
In Section~2, we describe our flexible version of \emph{jeu de taquin} as well as stating and proving
our main results. Section~3 uses the results of Section~2 to give
two additional formulations of the equivariant Littlewood-Richardson rule. We then
deduce Theorem~\ref{thm:main}. In Section~4, we formulate a conjectural formula for equivariant $K$-theory of Grassmannians. Concluding remarks are given in Section~5.
% equivariant insertion
% other types
% K_T

\section{Equivariant jeu de taquin (flexible version)}

To describe our flexible version of
equivariant \emph{jeu de taquin}, it is more convenient
to work with  semistandard fillings than with standard
fillings.

Starting with a semistandard filling $T$ of a skew shape
$\nu/\lambda$, choose an inner corner $\x$ and mark it with a $\bullet$.
We now define the
equivariant slide of $T$ into $\x$.
As in classical \emph{jeu de taquin},
the slide proceeds by a sequence of swaps, as the $\bullet$ moves through the
tableau.

However, the result of a slide is not necessarily
a single tableau, but rather a formal sum of tableaux, with coefficients
in $\mathbb Z[\beta_1,\dots,\beta_{n-1}]$, where $\beta_i=t_i-t_{i+1}$.
The way this arises in the course of the sequence of swaps is that sometimes
a swap will produce two tableaux.  One of them has no $\bullet$, and it
contributes directly to the output (with a coefficient),
while the other still has a $\bullet$, which we continue to swap.

\subsection{Definitions of the equivariant swaps} Let $\x\in T$ be as above.
Suppose $\y$ is the box to the immediate right of $\x$, and $\z$ is the box
immediately below $\x$. Let ${\mathfrak b}$ be the smallest neighbouring label below $\x$
(either the smallest one on the lower edge of $\x$ or the one in the box $\z$)
and let ${\mathfrak r}$ be
the label in $\y$.
\excise{If ${\mathfrak b}={\mathfrak r}$
we borrow the usual convention (from the classical semistandard theory) that ${\mathfrak b}$
takes precedence over ${\mathfrak r}$ for the purposes of our sliding rules below.}
Define ${\mathcal N}_{\x,{\mathfrak l}}^T$ to be the number of
occurences of a label ${\mathfrak l}$ in columns weakly to the right of the box $\x$ in $T$.

There are four kinds of swaps (I)--(IV) that we use:

\noindent
\emph{(I) ``vertical swap'': ${\mathfrak b}\leq {\mathfrak r}$ (or there is no ${\mathfrak r}$)
and ${\mathfrak b}$ is a box label of $\z$}: $T'$ is obtained by exchanging ${\bullet}$ and ${\mathfrak b}$, i.e.,
$T=\tableau{{\bullet}&{\mathfrak r}\\{\mathfrak b}}\mapsto \tableau{{\mathfrak b}&{\mathfrak r}\\{\bullet}}=T'$.

Output: $T'$.

\noindent
\emph{(II) ``expansion swap'': ${\mathfrak b}\leq {\mathfrak r}$
and ${\mathfrak b}$ is a label of the lower edge of $\x$}: $T'$ is obtained by moving ${\mathfrak b}$ into $\x$; the ${\bullet}$ is
eliminated. $T''$ is obtained by moving ${\mathfrak b}$ to the top edge of $\x$ (and ${\bullet}$ remains in place). In this case,
\[\begin{picture}(150,15)
\put(0,0){$\tableau{{\bullet}}\mapsto \mbox{$\beta(\x)$}\cdot\tableau{{\mathfrak b}}+\tableau{{\bullet}}=\beta(\x)\cdot T'+T''$}
\put(7,-3){$\mathfrak b$}
\put(108,14){$\mathfrak b$}
\end{picture}
\]

Output: \ $\mbox{$\beta(\x)$}\cdot T'+T''$.

\noindent
\emph{(III) ``resuscitation swap'': ${\mathfrak b}> {\mathfrak r}$ (or there is no ${\mathfrak b}$), and the largest label ${\mathfrak u}$ on the upper edge of $\x$
satisfies ${\mathfrak u}={\mathfrak r}$:}
%and semistandardness holds if ${\mathfrak u}={\mathfrak r}$
%moves down to replace the $\bullet$.
In this case, $T'$ is obtained by having ${\mathfrak u}={\mathfrak r}$
replace the ${\bullet}$ in $\x$,
replace ${\mathfrak r}$ by $\bullet$ in $\y$, and placing
${\mathfrak r}$ on the lower edge of $\y$. This move locally looks like:
\[\begin{picture}(170,18)
\put(0,0){$T=\tableau{{\bullet}&{\mathfrak r}}\mapsto
\tableau{{\mathfrak r}&{\bullet}}=T'$}
\put(31,16){$\mathfrak r$}
\put(107,-3){$\mathfrak r$}
\end{picture}
\]

Output: $T'$.

\noindent
\emph{(IV) ``horizontal swap'': ${\mathfrak b}>{\mathfrak r}$ (or there is no ${\mathfrak b}$),
and (III) does not apply}:
Define $Z$ to be the set of consecutive integers
$\{\mathfrak r, \mathfrak r+1,\dots, {\mathfrak m}\}$ where ${\mathfrak m}$ is
chosen largest so that:
\begin{itemize}
\item[(i)] ${\mathfrak m}<\mathfrak b$ and ${\mathfrak m}$ is at least as large as
the entry in the box to the left of $\x$;
\item[(ii)] ${\mathcal N}_{\y,{\mathfrak l}}^T= {\mathcal N}_{\y,{\mathfrak r}}^T$ for
all $\mathfrak r \leq \mathfrak l \leq {\mathfrak m}$.
\item[(iii)] $\{\mathfrak r+1,\dots,{\mathfrak m}\}$ are labels on the lower
edge of $\y$.
\end{itemize}

Set, $Z'=Z\setminus\{{\mathfrak m}\}$, $W=U\cup Z'$ and $Y'=Y\setminus Z$. Then locally the swap
is:
\begin{equation}
\label{eqn:IV}
\begin{picture}(100,18)
\put(0,0){$T=\tableau{{\bullet}&{\mathfrak r}}\mapsto
\tableau{{{\mathfrak m}}&{\bullet}}=T'$}
\put(31,15){$U$}
\put(31,-4){$\mathfrak b$}
\put(47,-4){$Y$}
\put(87,-4){$\mathfrak b$}
\put(105,-4){$Y'$}
\put(84,14){$W$}
\end{picture}
\end{equation}

That is, $T'$ is the result of moving ${\mathfrak m}$
into $\x$ and putting the smaller entries of $Z$ in the upper edge of $\x$. Conclude by placing
$\bullet$ into $\y$.

Output: $T'$.%\qed
\begin{Example}[of swap (IV)]\label{exa:IVswap} We have
\[\begin{picture}(100,15)
\put(50,0){$\tableau{{\ }&{\ }}$}
\put(57,6){$\bullet$}
\put(77,6){$1$}
\put(57,-4){$3$}
\put(71,-4){$2,3$}

\put(100,3){$\mapsto$}
\put(120,0){$\tableau{{\ }&{\ }}$}
\put(127,16){$1$}
\put(127,6){$2$}
\put(147,6){$\bullet$}
\put(127,-4){$3$}
\put(147,-4){$3$}
\end{picture}\]
where $Z=\{1,2\}$.

On the other hand:
\[\begin{picture}(100,30)
\put(0,18){$\tableau{{\ }&{\ }&{1}\\{\bullet }&{1}}\mapsto
\tableau{{\ }&{\ }&{1}\\{1}&{\bullet}}$}
\put(26,-6){$2$}
\put(101,-6){$2$}
\end{picture}\]
where the edge label ``$2$'' is not in $Z$ because of (IV)(ii).

In addition, the following swap (IV) \emph{is} valid, even though it ``breaks'' row semistandardness
in the ``obvious'' sense:
\[\begin{picture}(100,30)
\put(0,15){$\tableau{{\ }&{\ }&{\ }\\{\bullet }&{1}&{1 }}\mapsto
\tableau{{\ }&{\ }&{\ }\\{2}&{\bullet}&{1 }}$}
\put(26,-9){$2$}
\put(45,-9){$2$}
\put(82,12){$1$}
\put(121,-9){$2$}
\end{picture}\]
Note that the next swap will also be of type (IV), ``fixing'' the broken semistandardness
in the second row. Claims~\ref{claim:cannotbe} and \ref{claim:IVispossible} below explain how this example generalizes.
\qed
\end{Example}

We now describe ${\tt Eqjdt}_{\x}(T)$ (as opposed to the ``${\tt Ejdt}_{\x}(T)$'' of Section~1).
Begin by replacing $T$ by the result
of swapping at $\x$. The result is a formal sum of terms of the form $\omega \cdot S$
where $\omega\in {\mathbb Z}[\beta_1,\ldots,\beta_{n-1}]$, and $S$
is a tableau.
If a tableau $U$ in this formal sum either has no $\bullet$, or the $\bullet$ has no
neighbouring labels southeast, then do nothing. Otherwise, let $\x'$ be the box containing the $\bullet$
of $U$ and replace $U$ by swapping at $\x'$. Repeat until no more tableaux
need replacement. Now erase all any $\bullet$'s from the tableaux in the formal
sum. We need to show (under assumptions) that ${\tt Eqjdt}_{\x}(T)$ is a well-defined algorithm.

Call a tableau $T$ with at most a single $\bullet$
{\bf really good} if:
\begin{itemize}
\item[(a)] it is semistandard, once one ignores the $\bullet$ (i.e.,
the rows are weakly increasing and the columns are strictly increasing);
\item[(b)] the label of the box directly left of the box with the $\bullet$
is weakly less than the smallest label on the edge below the $\bullet$ (if the latter
label exists), i.e.,
\[\begin{picture}(200,15)
\put(50,0){$\tableau{{\ell }&{\bullet }}$}
\put(75,-4){$b$}
\put(100,5){$(\ell\leq b)$;}
\end{picture}\]
\item[(c)] the label of the box directly right of the box with the bullet is weakly larger
than the largest label on the edge above the $\bullet$ (if the latter exists), i.e.,
\[\begin{picture}(200,15)
\put(50,0){$\tableau{{\bullet }&{r }}$}
\put(55,16){$u$}
\put(100,5){$(u\leq r)$.}
\end{picture}\]
\end{itemize}
(Note that the latter two conditions would be automatic if the $\bullet$ were a numerical
label.) Call $T$ {\bf nearly bad} if (b) and (c) above hold, and (a) holds except
that the label to the immediate left of the $\bullet$ may be larger than the label
to the immediate right of $\bullet$. We will say $T$ is {\bf good} if it is either really good or nearly bad; otherwise $T$ is {\bf bad}.

\excise{
\begin{Example} The first three tableaux below are good:

$\ \ \ \ \ \ \ \ \ \ \ \ \ \ \ \ \ \begin{picture}(200,15)
\put(50,0){$\tableau{{\ }&{\ }}$}
\put(57,6){$\bullet$}
\put(77,6){$1$}
\put(57,-4){$1$}

\put(100,0){$\tableau{{\ }&{\ }}$}
\put(107,6){$\bullet$}
\put(127,6){$1$}
\put(107,14){$1$}

\put(150,0){$\tableau{{\ }&{\ }}$}
\put(157,6){$1$}
\put(177,6){$\bullet$}
\put(177,-4){$1$}

\put(200,5){but not:}

\put(250,0){$\tableau{{\ }&{\ }}$}
\put(257,6){$2$}
\put(277,6){$\bullet$}
\put(277,-4){$1$}

\put(300,0){$\tableau{{\ }&{\ }}$}
\put(307,6){$\bullet$}
\put(327,6){$1$}
\put(307,14){$2$}

\end{picture}$
\qed
\end{Example}}

The third swap in Example~\ref{exa:IVswap} demonstrates that swap (IV) can turn a really
good tableau to a nearly bad one. In fact, in Section~\ref{sec:proofthmab} we see only swap (IV) can cause near badness.

\subsection{Statement of the main results}
An equivariant filling $T$
is {\bf lattice} if for a given column $c$ and label ${\mathfrak l}$ (that may not be in column $c$),
the number of occurrences of ${\mathfrak l}$ in columns weakly to the right of column $c$ weakly exceeds
the occurences of ${\mathfrak l}+1$ in that region.

The appropriate class of tableaux to apply our ${\tt Eqjdt}$ swaps to are the lattice and
semistandard tableaux, in the sense that ${\tt Eqjdt}$ preserves this class:

\begin{Proposition}
\label{prop:semistandard}
Suppose $T$ is semistandard and lattice, and that $\x$ is an inner corner.
Then ${\tt Eqjdt}_{\x}(T)$ is well-defined as an algorithm: it
terminates in a finite number of steps, and outputs a
formal sum of semistandard and lattice tableaux. Each intermediate tableau in the calculation of
${\tt Eqjdt}_{\x}(T)$ is good and lattice.
\end{Proposition}

Assuming this proposition (the proof being delayed until Section~2.3), we define (an) {\bf equivariant rectification}. Given $T$, pick an inner corner $\x$
and replace $T$ by the formal sum ${\tt Eqjdt}_{\x}(T)$. Now, for each $U$ appearing in
${\tt Eqjdt}_{\x}(T)$, which has an inner corner $\x'$, replace $U$ by ${\tt Eqjdt}_{\x'}(U)$. Repeat until no such $U$ exists.
Let ${\tt Eqrect}(T)$
be the resulting formal sum of equivariant semistandard tableaux. We will call the choices of
$\x$ and of each $\x'$ the {\bf rectification order}.

    Call a straight shape tableau {\bf regular} if
does not have any edge labels; it is {\bf irregular} otherwise.
The regular tableau $S_{\mu}$ whose $i$-th row uses only the labels $i$ is called a {\bf highest
weight tableau}. The {\bf content} of a tableau $T$ is $\mu=(\mu_1,\mu_2,\dots)$ if $T$
has $\mu_1$ 1's, $\mu_2$ 2's, etc.  For $T$ of content $\mu$, ${\tt Eqrect}(T)$ is {\bf $\mu-$highest weight} if $S_{\mu}$ is
the only regular tableau that appears.  (We allow the possibility that
no regular tableau appears at all.)

Let us also define the {\bf a priori weight} of a good and lattice tableau $T$, denoted by
${\tt apwt}(T)$. Declare ${\tt apwt}(T)=0$ if:
\begin{itemize}
\item[(i)] there is an
edge label $i$ weakly above the upper edge of the box $\x$ in row $i$ (in its column), and it is not possible
to apply a resuscitation swap (III) to $T$ such that $i$ moves into $\x$; or
\item[(ii)] there is a box label $i$ located strictly higher than row $i$.
\end{itemize}
We will say that a label satisfying (i) or (ii) is {\bf too high}.  It will also be convenient to say that a label $i$ is {\bf nearly too high} if it lies on the upper edge
of a box in row $i$ but is not too high (i.e., a resuscitation swap (III) applies to $T$ and moves $i$ into $\x$).

Now suppose neither (i) nor (ii) holds.
Given an edge label $i$, suppose it lies on the lower edge of a box
$\x$ in row $r$. (If $i$ is on a top edge of $\Lambda$ then $r=0$.)
Define ${\tt apfactor}(i)$ as follows:
%If $i$ is {\bf nearly
%too high}, i.e., $i$ lies on the upper edge of a box $x$ in row $i$, but can be resuscitated using (III),
%then
%\begin{equation}
%\label{eqn:apfactorspecial}
%{\tt apfactor}(i)=t_{{\tt Man}(x)+1}-t_{{\tt Man}(x)+1+\mbox{$\#$ of $i$'s strictly to the right of $x$}},
%\end{equation}
\begin{equation}
\label{eqn:apfactorusual}
{\tt apfactor}(i)=t_{{\tt Man}(\x)}-t_{{\tt Man}(\x)+r-i+1+ \mbox{$\#$ of $i$'s strictly to the right of $\x$}}.
\end{equation}
where ${\tt Man}(\x)$ is the Manhattan distance as defined in Section~1.

Finally, let
\[{\tt apwt}(T)=\prod_{\mbox{\small $i$ is an edge label of $T$}}{\tt apfactor}(i).\]

We are now ready to state our main result, a partial analogue of the fundamental theorems of \emph{jeu de taquin}.

\begin{Theorem}
\label{thm:AB}
Let $T$ be a lattice semistandard tableau of content $\mu$.
Then:
\begin{itemize}
\item[(I)] ${\tt Eqrect}(T)$ is ${\mu}$-highest weight for any choice
of rectification order.
\item[(II)] The coefficient of $S_{\mu}$ in
${\tt Eqrect}(T)$ is invariant under these choices.
\item[(III)] The coefficient in (II) is ${\tt apwt}(T)$.
\end{itemize}
\end{Theorem}

\begin{Remark}
In the classical theory, $T$ rectifies to $S_{\mu}$ if and only if $T$ is lattice
and has content $\mu$. However, in our setting, analogues of these two conditions are
no longer equivalent. Specifically,
it is possible for a non-lattice tableau to become lattice using the
equivariant swaps.
For example, the starting tableau $T$ below is not lattice, but ${\tt Eqjdt}_{\x}(T)$ is:
\[\begin{picture}(250,35)
\put(0,20){$T=\tableau{{\ }&{\ }&{\ }\\{\bullet }&{2}&{2 }}\mapsto
\tableau{{\ }&{\ }&{\ }\\{2 }&{\bullet}&{2 }}
\mapsto
\tableau{{\ }&{\ }&{\ }\\{2 }&{2}&{\bullet }}={\tt Eqjdt}_{\x}(T)$}
\put(30,16){$1$}
\put(50,16){$1$}
\put(107,16){$1$}
\put(125,16){$1$}
\put(182,16){$1$}
\put(201,16){$1$}
\end{picture}\]
Therefore, we proceed to develop an equivariant Littlewood-Richardson rule
using the second of the two classically equivalent conditions.

In order to develop a rule using an analogue of the first condition, one needs swapping rules with the property that non-lattice fillings
stay non-lattice after a swap. It seems to us that such rules would be more complicated
than our current rules.
\qed
\end{Remark}

\begin{Example}
In the following rectification (inside $\Lambda=2\times 2$),
we suppress the computations concerning tableaux with labels that are too high (i.e., will rectify to a irregular tableau):
\[
\begin{picture}(450,50)
\put(26,25){$T=\tableau{{\ }&{\bullet}\\{\ }} \mapsto \beta_3 \ \tableau{{\ }&{1}\\{\bullet}}
+\tableau{{\ }&{\bullet}\\{\ }}$}
\put(57,0){$1$}
\put(75,20){$1$}
\put(127,0){$1$}
\put(181,0){$1$}
\put(199,40){$1$}
\put(220,25){$\mapsto \beta_3 \left( \beta_1 \ \tableau{{\bullet }&{1}\\{1}}
+\tableau{{\bullet }&{1}\\{\ }}\right)+\cdots$}
\put(332,22){$1$}
\end{picture}
\]
\[
\begin{picture}(450,50)
\put(95,25){$\mapsto \beta_1 \beta_3 \ \tableau{{1 }&{1}\\{\bullet }}
+\beta_3\left(\beta_2 \ \tableau{{1 }&{1}\\{\ }}+\tableau{{\bullet}&{1}\\{\ }}\right)+\cdots$}
\put(285,40){$1$}
\end{picture}
\]
\[
\begin{picture}(450,50)
\put(95,25){$\mapsto (\beta_1 \beta_3 +\beta_2\beta_3) \ \tableau{{1 }&{1}\\{\ }}
+\beta_3\ \tableau{{1 }&{\bullet}\\{\ }}+\cdots$}
\put(274,20){$1$}
\end{picture}
\]
\[
\begin{picture}(450,50)
\put(95,25){$\mapsto (\beta_1 \beta_3 +\beta_2\beta_3) \ \tableau{{1 }&{1}\\{\ }}
+\beta_3\ \left(\beta_3 \ \tableau{{1 }&{1}\\{\ }}+ \tableau{{1}&{\bullet}\\{ \ }}\right)+\cdots$}
\put(354,40){$1$}
\end{picture}
\]
\[
\begin{picture}(450,50)
\put(95,25){$\mapsto (\beta_1 \beta_3 +\beta_2\beta_3+\beta_3^2) \ \tableau{{1 }&{1}\\{\ }}
+\cdots={\tt Eqrect}(T)$}
\end{picture}
\]
Hence ${\tt Eqrect}(T)$ is ${(2)}$-highest weight.
Now, ${\tt apwt}(T)=(\beta_1+\beta_2+\beta_3)\beta_3$
which equals the coefficient of $S_{(2)}$ in ${\tt Eqrect}(T)$.
These two facts agree with parts (I) and (III) of Theorem~\ref{thm:AB}, respectively.
\qed
\end{Example}

\subsection{Proof of Proposition~\ref{prop:semistandard}} Suppose we start the computation of ${\tt Eqjdt}_{\x}(T)$, giving rise
to a sequence of swaps of tableaux:
\[T=T^{(0)}\mapsto T^{(1)}\mapsto \cdots\mapsto T^{(i)}.\]
(If we use swap (II) a ``branching'' occurs in the computation. The above sequence represents
one of the paths of the computation.)

We argue by induction that each successive tableau is good and lattice; the base case
is the hypothesis on $T$. If $T^{(i)}$ either has no $\bullet$ or no labels southeast of $\bullet$
then this is one of the tableaux appearing in ${\tt Eqjdt}_{\x}(T)$. Otherwise, we must show that
we can apply exactly one of the swaps (I)--(IV) to obtain $S'=T^{(i+1)}$ which is good and lattice.

There are two cases, depending on whether $S=T^{(i)}$ is really good or nearly bad.

\noindent
{\bf Case 1}: \emph{$S$ is really good:} We break our argument into several claims.

\begin{Claim} If it is possible to apply one of the swaps (I)--(IV) to $S$
then the result is good.
\end{Claim}
\begin{proof}
Suppose the vertical swap (I) is applied. Thus, $S$ locally looks like
\[S=\tableau{d&\bullet&e\\f&g&h},\]
where $g\leq e$ (and there is no label on the edge above the $g$). Thus we obtain
$S'=\tableau{d&g&e\\f&\bullet&h}$. To check that $S'$ is really good, one only needs
$d\leq g$ (if $d$ exists). If $d$ exists, so must $f$ and $d<f\leq g$ (since $S$ is good), as needed.

Next, suppose the expansion swap (II) is applied, thus
\[\begin{picture}(200,30)
\put(50,15){$S=\tableau{d&\bullet&e\\f&g&h}$}
\put(100,13){$y$}
\end{picture}
\]
where $y$ is the smallest label on its edge and $y\leq e$.
If $S'$ is the result of having the $y$ jump to the top edge, then $S'$ is really good since $S$ is
really good and, as we have assumed, $y\leq e$. Also, we know $d\leq y$ (again since $S$ is good)
and hence if $S'$ is the result of replacing $\bullet$ by $y$, then $S'$ is really good.

If a resuscitation swap (III) is used, we would have:
\[\begin{picture}(300,55)
\put(50,35){$S=\tableau{p&q&t\\m&\bullet&{\mathfrak r}\\s&f&h}$}
\put(100,13){${\mathfrak b}$}
\put(100,32){${\mathfrak r}$}
\put(150,35){$\mapsto S'=\tableau{p&q&t\\m&{\mathfrak r}&\bullet\\s&f&h}$}
\put(237,13){${\mathfrak r}$}
\put(217,13){${\mathfrak b}$}
\end{picture}\]
where ${\mathfrak u}={\mathfrak r}$ is the largest label on its edge.
Since $S$ is really good, $m\leq {\mathfrak r}$. Hence $S'$ is really good.

Finally, suppose we use a horizontal swap (IV) to arrive at $S'$. Thus:
\[\begin{picture}(300,35)
\put(50,15){$S=\tableau{d&\bullet&{\mathfrak r}&w\\s&f&h&v}$}
\put(99,11){${\mathfrak b}$}
\put(98,30){$U$}
\put(117,11){$Y$}
\put(160,15){$\mapsto S'=\tableau{d&{\mathfrak m}&\bullet&w\\s&f&h&v}$}
\put(244,11){$Y'$}
\put(227,11){${\mathfrak b}$}
\put(226,30){$W$}
\end{picture}\]
Recall $Z=\{{\mathfrak r},{\mathfrak r}+1,\ldots,{\mathfrak m}\}$ is the set of labels that move from the
third column to the second (relative to our local picture). Removal of these labels
clearly keeps the third column of $S'$
semistandard since the third column of $S$ is assumed to be semistandard.
By the really goodness of $S$ and the assumption that (III) does not apply, it follows that the maximal element immediately above the $\bullet$ in $S$ is
strictly less than $\mathfrak r$.  These considerations, and condition (IV)(i), imply the semistandardness of the second
column of $S'$.
%The semistandardness of the second column of $S'$
%is immediate from (IV)(i) and the condition (c) on really goodness of $S$ (CHECK).
Now, (IV)(i) allows, at worst,
the possibility that $S'$ is nearly bad, i.e., that $w<{\mathfrak m}$. However,
even in that case, $S'$ is good (by definition).
\end{proof}

\begin{Claim}
Exactly one of the swaps (I)--(IV) is applicable.
\end{Claim}
\begin{proof} In the case
${\mathfrak b}\leq {\mathfrak r}$ (or ${\mathfrak r}$ does not exist),
one can apply either a vertical or expansion swap but not both.
Thus suppose ${\mathfrak b}>{\mathfrak r}$ (or there is no ${\mathfrak b}$). Locally, we have
\[\begin{picture}(100,15)
\put(0,0){$S=\tableau{{a}&{\bullet}&{\mathfrak r}}.$}
\put(49,15){$\mathfrak u$}
\put(49,-3){${\mathfrak b}$}
\put(68,-3){$Y$}
\end{picture}
\]
(The argument is the same if ${\mathfrak b}$ is the label of the box below the $\bullet$.)
Since $S$ is good we have ${\mathfrak u}\leq {\mathfrak r}$.
If ${\mathfrak u}= {\mathfrak r}$ then one can apply
a resuscitation move (III) (and, by definition, not a horizontal swap (IV)).

Hence we may assume ${\mathfrak u}<{\mathfrak r}<{\mathfrak b}$. Now, (IV) is always possible since the set $Z$ in the definition of (IV) is nonempty by the given inequalities and the assumption
$a\leq {\mathfrak r}$ (since $S$ is really good).
\end{proof}

We also need to show $S'$ is lattice. This will be argued after Case~2 since the proof only assumes $S$ is good.

\noindent
{\bf Case 2}: \emph{$S$ is nearly bad:} In Case 1 we proved a really good tableau can
become nearly bad only after using swap (IV) on some tableau $S^{-}$.
Suppose then that $S$ was obtained using swap (IV) from some tableau $S^-$,
where $S^-$ may be nearly bad.
Let the local
pictures of these tableaux be
\[\begin{picture}(300,45)
\put(50,25){$S^{-}=\tableau{p&q&t&x\\d&\bullet&{\mathfrak r}&{{\overline{\mathfrak r}} }}$}
\put(107,3){${\mathfrak b}$}
\put(105,20){$U$}
\put(125,20){$\overline U$}
\put(124,3){$Y$}
\put(142,0){${\overline Y}$}

\put(180,25){$S=\tableau{p&q&t&x\\d&{\mathfrak m}&\bullet&{{\overline{\mathfrak r}} }}$}
\put(230,3){${\mathfrak b}$}
\put(228,20){$W$}
\put(248,20){$\overline U$}
\put(247,3){$Y'$}
\put(265,1){${\overline Y}$}

%\put(310,35){$S'=\tableau{p&q&t&x\\m&\bullet&{\mathfrak r}&{{\overline{\mathfrak r}} }}$}
%\put(360,13){${\mathfrak b}$}
%\put(358,30){$U$}
%\put(378,30){$\overline U$}
%\put(377,13){$Y$}
\end{picture}\]
where $S$ being nearly bad means ${\mathfrak m}>{\overline {\mathfrak r}}$. We now construct the next swap
$S\mapsto S'$.

\begin{Claim}
\label{claim:cannotbe}
No swap of type (I), (II), or (III) is applicable to $S$.
\end{Claim}
\begin{proof}
 If $Y'\neq \emptyset$
then let $b'=\min Y'$. Then $b'>{\mathfrak m}>{\overline {\mathfrak r}}$ (by column semistandardness of $S^{-}$ and the assumption $S$ is nearly bad). Therefore
(II) cannot be applied. If $Y'=\emptyset$ then a similar argument
shows that (I) cannot be applied either.
Also, if ${\overline u}=\max {\overline U}$ exists, then ${\overline u}<
\mathfrak r \leq \overline{\mathfrak r}$. Hence
a resuscitation swap (III) cannot be applied either.
\end{proof}

\begin{Claim}
\label{claim:IVispossible}
Swap (IV) is applicable to $S$.
\end{Claim}
\begin{proof}
%The proof of Claim~\ref{claim:cannotbe}
%show that if ${\mathfrak b}'$ is the (smallest) neighboring label
%below the $\bullet$ in $S$ (either in $Y'$ or in the box below the edge) then
%${\mathfrak b}'>{\mathfrak r}$.
Since $S^{-}$ is good, we have
$\max {\overline U}<{\mathfrak r}\leq {\overline {\mathfrak r}}$.
Also, by the definition of (IV) we have $\min Y'>{\mathfrak m}>{\overline{\mathfrak r}}$.
Therefore, ${\overline {\mathfrak r}}$ can be placed in the edge of ${\overline U}$
in $S$ and maintain the vertical semistandardness in that column. Thus, let ${\overline {{\mathfrak m}}}$
be the largest label from ${\overline Y}$ with this property such that the consecutive
sequence $\overline Z=\{{\overline{\mathfrak r}}, {\overline {\mathfrak r}}+1,\ldots, {\overline {\mathfrak m}}\}$
\emph{could} form the sequence of labels that move left in the swap (IV) starting from $S$.
That is, they satisfy (IV)(i),(ii),(iii) provided
${\overline {\mathfrak m}}\geq {\mathfrak m}$.
In this case, the swap (IV) $S\mapsto S'$
would result in a good tableau:
\[\begin{picture}(200,45)
\put(50,25){$S^{'}=\tableau{p&q&t&x\\d&{\mathfrak m}&{{\overline {{\mathfrak m}}}}&\bullet}$}
\put(102,3){${\mathfrak b}$}
\put(102,20){$W$}
\put(120,20){$\overline W$}
\put(119,2){$Y'$}
\put(138,0){${\overline Y}'$}
\end{picture}\]
In order to reach a contradiction, suppose ${\overline {{\mathfrak m}}}<{\mathfrak m}$.

    $S^{-}$ is lattice (by induction). Condition (IV)(ii) gives
    \begin{equation}
    \label{eqn:latticeequal1}
    {\mathcal N}_{{\rm col} \ 3, {\mathfrak r}}^{S^{-}}=
    {\mathcal N}_{{\rm col} \ 3, {\mathfrak r}+1}^{S^{-}}=\cdots
    ={\mathcal N}_{{\rm col} \ 3,{\mathfrak m}}^{S^{-}}.
    \end{equation}

Since the labels ${\mathfrak r},{\mathfrak r}+1,\ldots,{\mathfrak m}$ appear in column $3$
of $S^{-}$, (\ref{eqn:latticeequal1}) implies
\begin{equation}
\label{eqn:latticeequal2}
     {\mathcal N}_{{\rm col} \ 4, {\mathfrak r}}^{S}=
    {\mathcal N}_{{\rm col} \ 4, {\mathfrak r}+1}^{S}=\cdots
    ={\mathcal N}_{{\rm col} \ 4,{\mathfrak m}}^{S}.
    \end{equation}
We also know that
\begin{equation}
\label{eqn:someineq1}
{\mathfrak r}\leq \overline{\mathfrak r}\leq {\overline {{\mathfrak m}}}<{\mathfrak m}.
\end{equation}
The first inequality is the induction hypothesis:
$S^{-}$ is row semistandard (to the right of the $\bullet$). The second inequality is the vertical semistandardness in the fourth column of $S$ combined with the fact ${\overline {\mathfrak m}}\in {\overline Y}$. The third inequality is our assumption to be contradicted.

Suppose ${\overline {{\mathfrak m}}}+1(\leq {\mathfrak m})$ appears in column $4$ of $S$.
If ${\overline {\mathfrak m}}+1$ were in the box below the edge of ${\overline Y}$ in $S^{-}$ then
since $S^{-}$ is good,
the
box to its immediate left must be filled with ${\mathfrak q}\leq {\overline {{\mathfrak m}}}+1$. But ${\mathfrak m}\in Y$ and ${\mathfrak m}\geq {\overline {{\mathfrak m}}}+1\geq {\mathfrak q}$ implying this filling is impossible. Hence we may assume ${\overline {{\mathfrak m}}}+1\in {\overline Y}$.
Then this, together with
(\ref{eqn:latticeequal2}) and (\ref{eqn:someineq1}),
imply ${\overline {{\mathfrak m}}}+1(\leq {\mathfrak m}<\min Y')$ should have been included in $\overline Z$, contradicting the definition of
${\overline {{\mathfrak m}}}$.

Therefore ${\overline {{\mathfrak m}}}+1$ does not appear in column $4$ of $S^{-}$. Then let $X$ be the subtableau of $S^{-}$ using the boxes in columns weakly to the right of column $4$ of $S^{-}$. Then $X$ has the labels ${\mathfrak r},\ldots,{\overline {{\mathfrak m}}}+1$ in equal numbers, is lattice in those labels, and does not have ${\overline {{\mathfrak m}}}+1$ in its leftmost column. This is impossible, another contradiction. Hence, in fact, ${\mathfrak m}\leq {\overline {{\mathfrak m}}}$ as desired. This means
$S'$ is at worst nearly bad and therefore good, as desired.
\end{proof}

Summarizing, if $S$ is nearly bad then it was obtained by a horizontal swap (IV) from either a really good $S^{-}$ or a nearly bad $S^{-}$ whose near badness
occurs in the same row but one step to the right.

To complete both Cases 1 and 2, it remains to prove:

\begin{Claim}
Any swap $T\mapsto T'$ starting from a good and lattice $T$, results in $T'$ being lattice.
\end{Claim}
\begin{proof}
    None of the swaps (I), (II) nor (III) can turn a lattice tableau into a non-lattice
tableau, since
in each case the set of labels in each column stays the same.  Therefore,
suppose that a horizontal swap (IV) destroys latticeness.

Consider the local
diagram (\ref{eqn:IV}).  The labels that move from the second column
to the first column (with respect to our local diagram) are
$Z=\{{\mathfrak r},{\mathfrak r}+1,\ldots,{\mathfrak m}\}$.

The violation of latticeness must occur in the second column (and nowhere else in $T'$), since it is the
only column such that the multiset of entries weakly to its right has changed.
%However, by (IV)(i), none of the labels
%moved appear in the first column in $T$ (relative to the local diagram (\ref{eqn:IV})).
%Any of the labels in the first column in $T$
%remain there in $T'$ and the additional
%labels coming into that column from the right cannot affect their latticeness.
%
%Therefore the violation of latticeness occurs in the second column of $T'$.
An offending label
${\mathfrak l}+1$ (i.e., one such that ${\mathcal N}_{\y,{\mathfrak l}+1}^{T'}>{\mathcal N}_{\y,{\mathfrak l}}^{T'}$) is not weakly less than ${\mathfrak n}$ (the neighboring label
of $\bullet$ to  the north) since none of those labels moved. Also,
${\mathfrak l}+1\not\in Z$
since they do not appear in the second column of $T'$.
Moreover ${\mathfrak l}+1\leq {\mathfrak m}+1$
since the labels ${\mathfrak m}+1$ and larger have not moved. Thus the offending label must be ${\mathfrak l}+1={\mathfrak m}+1$,
i.e., ${\mathcal N}_{\y,{\mathfrak m}+1}^{T'}>{\mathcal N}_{\y,{\mathfrak m}}^{T'}$. Hence, there must be a ${\mathfrak m}+1$ in the second column.

We cannot have ${\mathfrak m}+1$ as a box label in the box immediately below $\y$, because then
the box label in the neighbor to the
left would also be ${\mathfrak m}+1$ (other values would violate the prerequisite (IV)(i) or
that $T$ is good).
Since ${\mathfrak m}$ does not already occur in the first column of $T$,
the assumption that $T'$ is not lattice implies $T$
fails the lattice condition for the label ${\mathfrak m}+1$,
at the first column, contrary to
our assumption that $T$ is lattice.

Therefore, ${\mathfrak m}+1$ is on the lower edge of $\y$.  Why
does it not lie in the set $Z$?  The reason must be failure of the
prerequisite (IV)(i) or (IV)(ii).
If it is condition (IV)(i), then there must already be a ${\mathfrak m}+1$ in the first
column, and again we conclude $T$ is not lattice, contrary
to our assumption.  If it violates condition (IV)(ii), then that means
${\mathcal N}_{\y,{\mathfrak m}+1}^T<{\mathcal N}_{\y,{\mathfrak r}}^T$ (since $T$ is lattice). However, since swap (IV) was used, by (IV)(ii) we see
\[{\mathcal N}_{\y,{\mathfrak m}}^T= {\mathcal N}_{\y,{\mathfrak r}}^T>
{\mathcal N}_{\y,{\mathfrak m}+1}^T.\]
Since ${\mathcal N}_{\y,{\mathfrak m}}^{T'}={\mathcal N}_{\y,{\mathfrak m}}^T -1$,
${\mathcal N}_{\y,{\mathfrak m}+1}^{T'}={\mathcal N}_{\y,{\mathfrak m}+1}^{T}$ and all the numbers involved
are integers, we have ${\mathcal N}_{\y,{\mathfrak m}}^{T'}\geq {\mathcal N}_{\y,{\mathfrak m}+1}^{T'}$ and
so ${\mathfrak m}+1$ satisfies the lattice condition in $T'$
after all, a contradiction.\end{proof}

Concluding, we have shown that after each swap we obtain a good and lattice tableau. Moreover,
given such a tableau, exactly one of the swaps (I)-(IV) is applicable. These swaps have the property
of either eliminating the $\bullet$, moving the $\bullet$ strictly east or south, or strictly decreasing
the number of labels southeast of the $\bullet$. Hence after a finite number of steps, each tableau
will have either no $\bullet$ or a single $\bullet$ on an outer corner (which can then be erased). Hence
the ${\tt Eqjdt}$ algorithm is well-defined and terminates as desired.
\qed

\subsection{Proof of Theorem~\ref{thm:AB}}
\label{sec:proofthmab}

Having established the well-definedness of ${\tt Eqjdt}$ in Proposition~\ref{prop:semistandard},
the next proposition is the remaining main step in our proof of Theorem~\ref{thm:AB}.

\begin{Proposition}
\label{prop:weightpres}
Let $T$ be a good and lattice tableau arising in the process of computing
${\tt Eqjdt}$ starting from
a semistandard and lattice tableau. If $T\mapsto T'$ is the
result of one of the swaps (I), (III) or (IV) then ${\tt apwt}(T)={\tt apwt}(T')$.
In the case of the expansion swap (II), if $T\mapsto \beta(\x)T'+T''$ then we have
${\tt apwt}(T)=\beta(\x){\tt apwt}(T')+{\tt apwt}(T'')$.
\end{Proposition}
\begin{proof}
We analyze each of the swaps (I)-(IV) in turn:

\noindent
\emph{Vertical swap (I):} Only the box label ${\mathfrak b}$ moves (up by one square). Hence if any label
was too high in $T$, it will also be too high in $T'$. So we may assume no label is too high in $T$. In addition, since
we use (I), no labels of $T$ are even nearly too high. Hence no labels in $T'$ other than perhaps ${\mathfrak b}$ can be
even nearly too high.
Thus, if ${\mathfrak b}$ is not too high in $T'$, then
the computation of each ${\tt apfactor}$ will be the same in $T$ and $T'$.

Suppose ${\mathfrak b}$ becomes too high in $T'$. Since the swap does not
destroy the lattice or goodness properties, there must be some $\mathfrak b-1$ to the
right of the $\mathfrak b$, which must therefore be strictly higher than the
new position of the $\mathfrak b$.  But this implies that the $\mathfrak b-1$
was too high in $T$, contrary to our assumption.

%Hence ${\mathfrak b}$ must have been
%in row ${\mathfrak b}$ in $T$.
%If there were any labels ${\mathfrak b}'$ above ${\mathfrak b}$ in $T$ (and in the same column), it must be that
%${\mathfrak b}'<{\mathfrak b}$ (since $T$ is good). However, then ${\mathfrak b}'$ must have been too high
%in $T$, since the lowest ${\mathfrak b}'$ can be is the upper edge of the box in row
%${\mathfrak b}-1$, and since we did not do a resuscitation. This is
%a contradiction. Thus no labels can be above ${\mathfrak b}$ (and in the same column). Our assumptions then imply ${\mathfrak b}$ is in row ${\mathfrak b}$ of $T$ and
%moves into row ${\mathfrak b}-1$ of $T'$.

\excise{
Since $T$ is good and we used (I) then $T'$ is really good (by the proof of Proposition~\ref{prop:semistandard}). Thus, in $T'$
all labels weakly to the right of the ${\mathfrak b}$
(in box $x$ of $T'$) that satisfy
${\mathfrak l}<{\mathfrak b}$ are actually strictly north.
However $T$ is lattice implies $T'$ is lattice.
So, $T'$ being both lattice and good implies that some ${\mathfrak l}$ must
appear in a strictly higher row than some ${\mathfrak l}+1$ for each ${\mathfrak l}=1,2,\ldots,{\mathfrak b-1}$. This requires at least ${\mathfrak b}-1$ rows strictly
above the ${\mathfrak b}$ in $T'$, but there are only
${\mathfrak b}-2$ such rows, a contradiction. Hence ${\mathfrak b}$ is not too high in
$T'$.}

Since the edge labels are in the same positions in $T$ and $T'$, it now follows
that ${\tt apwt}(T)={\tt apwt}(T')$, as desired.

\noindent
\emph{Expansion swap (II):} Recall $T'$ is the tableau obtained by moving ${\mathfrak b}$ into the box $\x$, ``emitting
the weight'' $\beta(\x)$, whereas $T''$ is the tableau obtained by ${\mathfrak b}$ ``jumping over'' $\x$. Thus, if $T$ has any
labels that are too high, this will be true of both $T'$ and $T''$, in which case
\[0={\tt apwt}(T)=\beta(\x){\tt apwt}(T')+{\tt apwt}(T'')=\beta(\x)\cdot 0+0,\]
as desired. Hence we may assume no labels of $T$ are too high.

\noindent
{\bf Case 1:} \emph{The ${\mathfrak b}$ in $T''$ is not too high:} Note that the ${\mathfrak b}$ in $T'$ is also
not too high: we could only have $\mathfrak b$ too high in $T'$ if
$\mathfrak b$ was at the top edge of the box in row ${\mathfrak b}$ in $T$. However, since it was not resuscitated, it
would have been too high in $T$, a contradiction. Thus the highest ${\mathfrak b}$ can be in $T'$ is row ${\mathfrak b}$.
The ${\tt apfactor}$ of all edge labels other than the ${\mathfrak b}$ are the same in $T, T'$ and $T''$. (No label could
become nearly too high
in $T''$ except possibly ${\mathfrak b}$.)
So, it remains to prove that
\begin{equation}
\label{eqn:lastfactor}
{\tt apfactor}_T({\mathfrak b})=\beta(\x)+{\tt apfactor}_{T''}({\mathfrak b})
\end{equation}

Since the box above ${\mathfrak b}$ in
$T''$ has Manhattan distance ${\tt Man}(\x)+1$, we have by (\ref{eqn:apfactorusual}) that
\[{\tt apfactor}_{T''}({\mathfrak b})=t_{{\tt Man}(\x)+1}-t_{{\tt Man}(\x)+1+(r-1)-{\mathfrak b}+1+\mbox{$\#$ of ${\mathfrak b}$'s strictly to the right
 of $\x$ in $T''$}}.\]
But the number of ${\mathfrak b}$'s strictly to the right of our ${\mathfrak b}$ in $T''$ equals the number of
${\mathfrak b}$'s strictly to the right of ${\mathfrak b}$ in $T$. Thus, since $\beta(\x)=t_{{\tt Man}(\x)}-t_{{\tt Man}(\x)+1}$,
(\ref{eqn:lastfactor}) follows immediately.

%(Subcase 1b: The ${\mathfrak b}$ in $T''$ is nearly too high): Thus ${\mathfrak b}$ can be resuscitated at the very
%next step. Therefore by (\ref{eqn:apfactorspecial}):
%\[{\tt apfactor}_{T''}({\mathfrak b})=t_{{\tt Man}(x)+1}-t_{{\tt Man}(x)+1+\mbox{$\#$ of ${\mathfrak b}$'s strictly to the right of $x$ in $T''$}}.\]
%Thus (\ref{eqn:lastfactor}) is true, for the same reason as in Subcase 1a.

\noindent
{\bf Case 2}: \emph{The ${\mathfrak b}$ in $T''$ is too high:}
Since ${\mathfrak b}$ is not too high in $T$,
${\mathfrak b}$ in $T''$
is on the upper edge of a box in row ${\mathfrak b}$. This label is too high
because it cannot be resuscitated. Consider the box $\y$ to the immediate right of $\x$ in
$T$ (or $T''$). For us to have done an expansion step $T\to T''$, if there is a label ${\mathfrak r}$ in $\y$ of $T$, it
must satisfy ${\mathfrak r}\geq {\mathfrak b}$. However, if ${\mathfrak r}>{\mathfrak b}$ then since
$\y$ is in row ${\mathfrak b}$ we can conclude ${\mathfrak r}$ is too high in $T$, a contradiction of
our assumption about $T$.

Thus $\y$ either has no box label (explaining why we can't do a resuscitation) or $\y$ contains
${\mathfrak b}$. If the former is true then ${\tt apfactor}_T({\mathfrak b})=\beta(\x)$ since there can be no
${\mathfrak b}$'s strictly to the right of $\x$ (by the goodness and highness assumptions). So we are done in this situation.
Hence assume $\y$ contains ${\mathfrak b}$. Thus the resuscitation swap (III) was possible after all in $T''$, contradicting our assumption that the ${\mathfrak b}$ in $T''$ is too high.

\noindent
\emph{Resuscitation swap (III):}
Only two labels move, namely ${\mathfrak u}={\mathfrak r}$ and ${\mathfrak r}$ go downwards.
Suppose a label ${\mathfrak n}$ on the top edge of box $\y$ in $T$ is too high but
becomes only nearly too high in $T'$. Hence $\y$ must be in row ${\mathfrak n}$. But
since swap (III) was applied, by semistandardness, ${\mathfrak n}<{\mathfrak u}={\mathfrak r}$.
Hence ${\mathfrak u}={\mathfrak r}$ must be too high in both $T$ and $T'$ and thus
${\tt apwt}(T)={\tt apwt}(T')=0$. Thus we may assume this does not happen.
It is therefore clear that no other labels, except possibly ${\mathfrak u}={\mathfrak r}$ and
${\mathfrak r}$ can be too high in $T$ and become not too high in $T'$.
Hence we assume all labels of $T$ except possibly ${\mathfrak u}={\mathfrak r}$ and ${\mathfrak r}$
are not too high.

If ${\mathfrak u}={\mathfrak r}$ is on the upper edge of a box in row ${\mathfrak u}$ then it must be only nearly
too high in $T$, since by assumption we can apply a resuscitation swap (III) to bring that label downwards. If ${\mathfrak u}={\mathfrak r}$ were any higher in $T$ (and thus too high), it would be still too
high in $T'$. Thus we may assume it was not too high in $T$.
Thus ${\mathfrak r}$ must be not too high.
Summarizing, we can assume that no label of $T$ is too high.

Since ${\mathfrak u}={\mathfrak r}$ and ${\mathfrak r}$ move down when $T\to T'$,  no labels of
$T'$ are too high. Since the set of labels in each column is the same, it follows that
${\tt apwt}(T)={\tt apwt}(T')$ provided that
\begin{equation}
\label{eqn:resuscitationtodo}
{\tt apfactor}_T({\mathfrak u})={\tt apfactor}_{T'}({\mathfrak r}).
\end{equation}
(Recall we argued above that no labels other than ${\mathfrak u}$ can be
nearly too high in $T$ or $T'$.)

%\underline{{\bf Case 1}: ${\mathfrak u}$ is not nearly too high in $T$.}
There is a box above $\x$, say ${\sf w}$.
By (\ref{eqn:apfactorusual}) we have
\[{\tt apfactor}_T({\mathfrak u})=t_{{\tt Man}({\sf w})}-t_{{\tt Man}({\sf w})+{\tt row}({\sf w}) -{\mathfrak u}+1+\Delta_{\x,{\mathfrak u}}^{T}},\]
where $\Delta_{\x,{\mathfrak u}}^T$ is the number of ${\mathfrak u}$'s strictly to the right of $\x$ in $T$.

Also by (\ref{eqn:apfactorusual})
\[{\tt apfactor}_{T'}({\mathfrak r})=t_{{\tt Man}({\y})}-t_{{\tt Man}({\y})+
{\tt row}({\y})-{\mathfrak r}+1+\Delta_{\y,{\mathfrak r}}^{T'}}.\]
where $\Delta_{\y,{\mathfrak r}}^{T'}$ is the number of ${\mathfrak r}$'s strictly to the right of
${\y}$ in $T'$.

Noting that
\begin{eqnarray}\nonumber
{\tt Man}({\sf w}) & = & {\tt Man}({\y})\\ \nonumber
{\tt row}({\sf w}) & = & {\tt row}({\y})-1\\ \nonumber
\Delta_{\y,{\mathfrak r}}^{T'} & = & \Delta_{\x,{\mathfrak u}={\mathfrak r}}^{T}-1
\end{eqnarray}
we conclude (\ref{eqn:resuscitationtodo}) is true.

%\noindent
%\underline{{\bf Case 2}: ${\mathfrak u}$ is nearly too high in $T$.} By (\ref{eqn:apfactorspecial}),
%\[{\tt apfactor}_T({\mathfrak u})=t_{{\tt Man}(x)+1}-t_{{\tt Man}(x)+1+
%\Delta_{x,{\mathfrak u}}^{T}}.\]
%We have ${\mathfrak r}={\mathfrak u}$ and ${\mathfrak r}$ is not too high
%nor nearly too high in $T'$. Thus by
%(\ref{eqn:apfactorusual})
%\[{\tt apfactor}_{T'}({\mathfrak r})=t_{{\tt Man}({y})}-t_{{\tt Man}(y)-{\mathfrak u}+{\mathfrak u}+1+\Delta_{y,{\mathfrak u}}^{T'}}.\]
%Then since
%\begin{eqnarray}\nonumber
%{\tt Man}({y})&=& {\tt Man}(x)+1\\ \nonumber
%\Delta_{y,{\mathfrak u}}^{T'} & = & \Delta_{x,{\mathfrak u}}^{T}-1\nonumber
%\end{eqnarray}
%we obtain ${\tt apfactor}_T({\mathfrak u})={\tt apfactor}_{T'}({\mathfrak u})$, as needed.

\noindent
\emph{Horizontal swap (IV):} If any label of $T$ is too high then since labels are
moving weakly upwards, that label will also be too high in $T'$. Thus, we may assume that no label
of $T$ is too high. We did not resuscitate ${\mathfrak u}=\max U$, nor labels on the upper edge of $\y$. Hence labels on
these edges are not even nearly too high in $T$.

Recall $Z=\{{\mathfrak r},{\mathfrak r}+1,\ldots,{\mathfrak m}\}$ are the labels that moved during the swap (IV). If ${\mathfrak m}={\mathfrak r}$, then ${\mathfrak r}$ is the only label that moves, and moreover
it simply moves directly to the left from box $\y$ to box $\x$. So if ${\mathfrak r}$ was not
too high in $T$, nor is it too high in $T'$. Next suppose ${\mathfrak m}>{\mathfrak r}$. Now ${\mathfrak r}$
moves into the upper edge of $\x$. Since ${\mathfrak r}+1\in Y$ and ${\mathfrak r}+1$ is not too high in $T$, we see $\x$ and $\y$ are in row $R\geq {\mathfrak r}+1$. Similarly, in fact $\x$ and $\y$ must be in row $R\geq {\mathfrak m}$, since otherwise ${\mathfrak m}$ would be too high.
Consequently, in $T'$, all of the labels in $Z$ are still not too high.

We also need to rule out the possibility that an edge-label which is too high in $T$ could become
nearly too high in $T'$ (because the next step after $T'$ would be a swap (III) resuscitating it).
Suppose locally the picture looks like
\begin{equation}
\label{eqn:depicted}
\begin{picture}(250,25)
\put(28,8){$T=$}
\put(50,5){$\tableau{{\bullet}&{\mathfrak r}&{\overline{\mathfrak r}}}$}
\put(74,0){$Y$}
\put(76,20){$\overline{\mathfrak r}$}
\put(55,20){$U$}

\put(120,10){$\mapsto$}
\put(150,5){\tableau{{\mathfrak m}&{\bullet}&{\overline{\mathfrak r}}}}
\put(174,0){$Y'$}
\put(176,20){$\overline{\mathfrak r}$}
\put(155,20){$W$}
\put(210,8){$=T'$}
\end{picture}
\end{equation}
If $T'$ is really good then ${\mathfrak m}\leq {\overline{\mathfrak r}}$, but ${\mathfrak m}>{\overline{\mathfrak r}}$ by the vertical semistandardness of $T$, a contradiction.
Otherwise if $T'$ is nearly bad, then the next swap is (IV) not (III), by Claim~\ref{claim:cannotbe}. Thus, again $T'$ cannot be of the form in (\ref{eqn:depicted}).

Consider any edge label $i$ that did not change in the swap $T\to T'$. Note ${\tt apfactor}(i)$
and ${\tt apfactor}_{T'}(i)$ could only differ if the number of $i$'s strictly to the right of
the given $i$ changes as we compare $T$ and $T'$. However, there could not be a nonzero change,
by the definition of the swap (IV).

We now establish a weight-preserving correspondence between the edge labels of
$T$ which moved and the edge labels of $T'$ which resulted from the move;
specifically,
\[{\tt apfactor}_{T'}({\mathfrak l})={\tt apfactor}_T({\mathfrak l}+1)\]
for ${\mathfrak l}={\mathfrak r},{\mathfrak r}+1,\ldots,{\mathfrak m}-1$, using (\ref{eqn:apfactorusual}).
To see this, first note that in each case ${\mathfrak l}$ is in a one higher row in $T'$ than in $T$. Therefore it remains to show
\begin{eqnarray}
\label{eqn:finalequalityIV}
{\mathcal N}_{\y,{\mathfrak l}+1}^T={\mathcal N}_{\x,{\mathfrak l}}^{T'}.
\end{eqnarray}

Now by (IV)(ii) we have
\begin{equation}
\label{eqn:weconclude1}
{\mathcal N}_{\y,{\mathfrak r}}^T={\mathcal N}_{\y,{\mathfrak r}+1}^T=\cdots={\mathcal N}_{\y,{\mathfrak m}}^T.
\end{equation}
Finally, by (IV)(i), we know there were no ${\mathfrak l}$'s in the column of $\x$ in $T$, so
\begin{equation}
\label{eqn:weconclude2}
{\mathcal N}_{\x,{\mathfrak l}}^{T'}={\mathcal N}_{\y,{\mathfrak l}}^T.
\end{equation}
Now (\ref{eqn:weconclude1}) and (\ref{eqn:weconclude2}) combined
immediately gives (\ref{eqn:finalequalityIV}).
\end{proof}

\noindent
\emph{Conclusion of the proof of Theorem~\ref{thm:AB}:}
By Proposition~\ref{prop:semistandard}, any tableau in ${\tt Eqrect}(T)$ (under any rectification order) is semistandard and lattice. The only regular, semistandard, lattice tableaux of straight shape are the highest weight tableaux.
Since $T$ is lattice then ${\tt Eqrect}(T)$ (with respect to any order) will be a sum of tableaux that are
lattice and which have the same multiset of labels as $T$.
%If all of these tableaux were irregular, then that would violate the preservation ${\tt apwt}(T)$ given by Proposition~\ref{prop:weightpres},
%since each irregular tableau $S$ has ${\tt apwt}(S)=0$.
Hence the only regular tableau that can appear is $S_\mu$. Any irregular $U$ that appears in ${\tt Eqrect}(T)$ has ${\tt apwt}(U)=0$. Hence,
by Proposition~\ref{prop:weightpres},
  the coefficient of $S_{\mu}$ in ${\tt Eqrect}(T)$ is ${\tt apwt}(T)$ and the theorem holds. \qed

\section{Equivariant jeu de taquin computes Schubert calculus}

Let
\[D_{\lambda,\mu}^{\nu}=\sum_T [S_{\mu}] \ {\tt Eqrect}(T)=\sum_T {\tt apwt}(T)\]
where the sums are over all lattice and semistandard tableaux $T$
of shape $\nu/\lambda$ and content $\mu$ such that ${\tt Eqrect}(T)$ is ${\mu}$-highest weight. (By the arguments of Section~2, the last condition can be replaced
by ${\tt apwt}(T)\neq 0$.) Also, here
$[S_{\mu}] \ {\tt Eqrect}(T)$ means the coefficient
of $S_{\mu}$ under some (or, as we proved in Theorem~\ref{thm:AB}, any) rectification order.
(The second equality is Theorem~\ref{thm:AB}(III).)

We now connect these polynomials to the Schubert structure coefficients:
\begin{Theorem}
\label{thm:DequalsC}
$D_{\lambda,\mu}^{\nu}=C_{\lambda,\mu}^{\nu}$
\end{Theorem}

The ${\tt Eqrect}$ method of computing $D_{\lambda,\mu}^{\nu}$ generates each monomial
of this polynomial (as expressed in the variables $\beta_i$) separately. This is somewhat
different than other rules for these polynomials, which express the answer (as the ${\tt apwt}$
computation does) by combining many of these monomials into one.

Our proof follows the same general strategy used in \cite{Knutson.Tao}.
However the technical details are, naturally, significantly different.
Although we can state the rule $D_{\lambda,\mu}^{\nu}=\sum_T {\tt apwt}(T)$ without
development of ${\tt Eqjdt}$, our proof relies on this construction.

\begin{Proposition}
\label{prop:DequalsCspecial}
$D_{\lambda,\mu}^{\lambda}=C_{\lambda,\mu}^{\lambda}$.
\end{Proposition}

We delay the proof of the above proposition until after the proof of Theorem~\ref{thm:DequalsC}.

For completeness, we restate and prove the following recurrence from \cite[Proposition~3.4]{Molev.Sagan}
and also observed by A.~Okounkov; see also \cite[Proposition~2]{Knutson.Tao}.
\begin{Lemma}
\label{lemma:therecC}
We have
\begin{equation}
\label{eqn:therecC}
\sum_{\lambda^{+}} C_{\lambda^{+},\mu}^{\nu}=C_{\lambda,\mu}^{\nu}\cdot {\tt wt}(\nu/\lambda)
+\sum_{\nu^{-}}C_{\lambda,\mu}^{\nu^{-}}
\end{equation}
where
\begin{itemize}
\item $\lambda^{+}$ is obtained by adding an outer corner to $\lambda$;
\item $\nu^{-}$ is obtained by removing an outer corner of $\nu$; and
\item ${\tt wt}(\nu/\lambda)=\sum_{\x\in \nu/\lambda} \beta(\x)$.
\end{itemize}
\end{Lemma}
\begin{proof}
The equivariant Pieri rule states
\begin{equation}
\label{eqn:equivPieri}
\sigma_{(1)}\cdot\sigma_{\lambda}=\sum_{\lambda^+}\sigma_{\lambda^+}+{\tt wt}(\lambda)\sigma_{\lambda}\in H_T(X).
\end{equation}
Equation~(\ref{eqn:equivPieri}) is proved in \cite[Proposition~2]{Knutson.Tao}. To repeat the argument, it
follows from the classical Pieri rule combined with the localization computation
$C_{\lambda,(1)}^{\lambda}={\tt wt}(\lambda)$; this localization computation
is easily recovered from the earlier results discussed in Section~3.1. Hence
\begin{eqnarray}\nonumber
\sigma_{(1)}\cdot(\sigma_{\lambda}\cdot\sigma_{\mu}) & = &  \sigma_{(1)}\cdot\left(\sum_\nu C_{\lambda,\mu}^\nu\sigma_{\nu}\right) =  \sum_{\nu}C_{\lambda,\mu}^\nu \sigma_{(1)}\cdot\sigma_{\nu}\\ \nonumber
& = & \sum_{\nu} C_{\lambda,\mu}^{\nu}{\tt wt}(\nu)\sigma_{\nu}+ \sum_\nu C_{\lambda,\mu}^{\nu} \sum_{\nu^+}\sigma_{\nu^+}.\nonumber
\end{eqnarray}
Also,
\begin{eqnarray}\nonumber
(\sigma_{(1)}\cdot\sigma_{\lambda})\cdot\sigma_{\mu} & = & \left({\tt wt}(\lambda)\sigma_{\lambda}+\sum_{\lambda^+}\sigma_{\lambda^+}\right)\cdot\sigma_{\mu}= {\tt wt}(\lambda)\sigma_{\lambda}\cdot \sigma_{\mu} +\sum_{\lambda^+}\sigma_{\lambda^+}\cdot\sigma_{\mu}\\ \nonumber
& = & {\tt wt}(\lambda)\sum_{\nu} C_{\lambda,\mu}^{\nu}\sigma_{\nu}+\sum_{\lambda^+}\sum_{\nu} C_{\lambda^+,\mu}^{\nu}\sigma_{\nu}.
\end{eqnarray}
Now, $\sigma_{(1)}\cdot(\sigma_{\lambda}\cdot\sigma_{\mu})=(\sigma_{(1)}\cdot\sigma_{\lambda})\cdot\sigma_{\mu}$
since $H_T$ is an associative ring. Thus taking the coefficient of $\sigma_{\nu}$ on both sides
of this identitiy gives the conclusion.
\end{proof}

\noindent
\emph{Proof of Theorem~\ref{thm:DequalsC}:} Suppose that $\{D_{\lambda,\mu}^\nu\}$ satisfies
\begin{equation}
\label{eqn:therec}
\sum_{\lambda^{+}} D_{\lambda^{+},\mu}^{\nu}=D_{\lambda,\mu}^{\nu}\cdot {\tt wt}(\nu/\lambda)
+\sum_{\nu^{-}}D_{\lambda,\mu}^{\nu^{-}}
\end{equation}
and we have established Proposition \ref{prop:DequalsCspecial} (as done in Section~3.1). Then,
by induction on $|\nu|-|\lambda|\geq 0$, the recurrence
(\ref{eqn:therec}) together with the initial condition $D_{\lambda,\mu}^{\lambda}=C_{\lambda,\mu}^{\lambda}$ uniquely determine $D_{\lambda,\mu}^{\nu}$; cf.~\cite[Corollary 1]{Knutson.Tao}. Hence, by Lemma~\ref{lemma:therecC} it follows that $D_{\lambda,\mu}^{\nu}=C_{\lambda,\mu}^{\nu}$. This would complete the proof of the theorem.

Hence it remains to show that the polynomials
$\{D_{\lambda,\mu}^\nu\}$ satisfy (\ref{eqn:therec}). Let ${\mathcal D}_{\lambda,\mu}^{\nu}$
denote the set of witnessing lattice and semistandard tableaux that rectify to $S_{\mu}$. Fix $\lambda^{+}$
and consider $T\in {\mathcal D}_{\lambda^{+},\mu}^{\nu}$. Let $\x=\lambda^{+}/\lambda$
and consider the tableaux $\{S:[S] \ {\tt Eqjdt}_{\x}(T)\neq 0\}$.
Among these $S$, exactly one is of shape $\nu^{-}/\lambda$ (for some $\nu^{-}$). For this $S$
we have $\omega_S=1$ and $S\in {\mathcal D}_{\lambda,\mu}^{\nu^-}$.
The other $S$ appearing in the formal sum arise from an expansion of an
edge label into a box $\y$ in $\nu/\lambda$ and $\omega_S=\beta(\y)$; also $S\in {\mathcal D}_{\lambda,\mu}^\nu$. By construction, no other kinds of tableaux can appear. (In this paragraph, we have tacitly used Proposition~\ref{prop:semistandard}.)

It remains to show that:
\begin{itemize}
\item[(a)] Given $W\in {\mathcal D}_{\lambda,\mu}^{\nu^-}$
there is a unique $\lambda^+$ and a unique $T\in {\mathcal D}_{\lambda^+,\mu}^\nu$ such that \[[W] \ {\tt Eqjdt}_{\x}(T)=1.\]
\item[(b)] Given $W\in {\mathcal D}_{\lambda,\mu}^{\nu}$ and a box ${\sf b}\in \nu/\lambda$
there is a unique $\lambda^+$ and a unique $T\in {\mathcal D}_{\lambda^+,\mu}^\nu$
such that
\[[W] \ {\tt Eqjdt}_{\x}(T)=\beta({\sf b}).\]
\end{itemize}

In order to prove (a) and (b), we need to develop a notion of reverse {\tt Eqjdt}.
In (a),
we wish to argue that from $W$ and the
box $\bbb=\nu/\nu^-$ there is a unique sequence of
tableaux
\begin{equation}
\label{eqn:uniqueseq}
T=U^{(-N)}\mapsto\cdots \mapsto U^{(-1)}\mapsto U^{(0)} = W,
\end{equation}
(for some $N$) where each $U^{(-j)}$ is a good and lattice
tableau. Moreover, $U^{(-j)}\mapsto U^{(-j+1)}$ means $U^{(-j+1)}$
is obtained from $U^{(-j)}$ by one of the swaps (I)-(IV) into the box of $U^{(-j)}$ containing
the $\bullet$. In (b) we wish to make the same argument, except that $U^{(0)}$ is
obtained from $W$ by moving the label in ${\sf b}$ to the lower edge of ${\sf b}$, and
a $\bullet$ is placed in ${\sf b}$.

\excise{(At this point, we pause to emphasize to the reader that, in principle,
the same general strategy is applicable to
proving the first version of the Littlewood-Richardson rule given in the introduction. However, we were
unable to construct proofs of (a) and (b) above in that context. In some sense, the main problem is that our
first version of {\tt Eqjdt} does not allow for arbitrary rectification orders.)}

Now, (a) and (b) follow from three claims.

\begin{Claim}
\label{claim:good}
Suppose $U=U^{(-i)}$ is a really good and lattice tableau with $\bullet$ in box ${\sf b}$ and locally near ${\sf b}$ we label the boxes as $U=\tableau{\cdots&{\sf a}\\{\sf c}&{\sf b}}$.
If box ${\sf a}$ or box ${\sf c}$ has a label, or if the upper edge of ${\sf b}$ has a label, then there exists a unique good and lattice
tableau $V$ with $\bullet$ in box ${\sf d}\in \{{\sf a},{\sf b},{\sf c}\}$ such that $V\to U$, using one of the swaps (I)--(IV).
\end{Claim}

\noindent
\emph{Proof of Claim~\ref{claim:good}:} There are two main cases, depending on whether the
upper edge of ${\sf b}$ is empty or not.

\noindent
{\bf Case 1}: \emph{Locally $U$ looks like $\tableau{{z}&y&w\\{x}&\bullet&q}$, where the upper edge of the box ${\sf b}$
containing $\bullet$ is empty, but other edges are possibly nonempty.}

\noindent
(Subcase 1a: $x\leq y$ or $x$ does not exist): Since $U$ is good, we have $z<x\leq y\leq w<q$.
If $V=\tableau{{z}&\bullet&w\\x&y&q}$ then $V$ is (really) good and also lattice since $U$ is lattice. Moreover,
since $y\leq w$ then we can apply the vertical swap (I) to give $U$.
Hence it remains to show that there are no other possible choices of $V$.

Clearly a expansion swap (II) could not result in $U$ since
we assume the edge immediately above the $\bullet$ in $U$ is empty. Also, swaps (III) and (IV)
are not possible if $x$ does not exist. Thus, we assume $x$ exists.

If resuscitation (III) results in $U$ then the box with $x$ in $U$ had a
$\bullet$ in $V$, and the ${\mathfrak u}=x$ is on the top edge of this box in $V$.
But $y\geq x$ implies $V$ is not semistandard in the second column.

Finally, if a horizontal swap (IV) resulted in $U$, then
\[\begin{picture}(100,35)
\put(20,15){$V=\tableau{{z}&y&w\\\bullet&{\mathfrak r}&q}$}
\put(69,-8){$Y$}
\end{picture}
\]
where $x\in \{{\mathfrak r}\}\cup Y$. However,
since $x\leq y$, we have a violation of
vertical semistandardness in the second column of $V$. Hence, (IV) could not have used either.

\noindent
(Subcase 1b: $x>y$, or $y$ does not exist):
If $y$ does not exist then clearly the vertical swap (I) did not result in $U$. If
$y$ exists then the same is true since
we would have $V=\tableau{{z}&\bullet&w\\x&y&q}$: but since $x>y$ then we obtain a
violation of semistandardness in the second row.

As in subcase 1a, the expansion swap (II) cannot produce $U$ since we have assumed that the edge
directly above the $\bullet$ in $U$ is empty.

Resuscitation (III) can happen if
\[\begin{picture}(100,35)
\put(20,15){$V=\tableau{{z}&y&w\\\bullet&x&q}$}
\put(51,13){$x$}
\end{picture}
\]
and $x$ is the least label
 in the edge below the $\bullet$ in $U$.
Note $V$ is good and lattice since $U$ has these
properties. Clearly, there is at most one way to reverse using (III).

 On the other hand, if a reversal using (III) is not possible, then we aim to construct a
horizontal swap $V'\mapsto U$ where
\[
\begin{picture}(100,35)
\put(20,15){$V'=\tableau{z&y&w\\\bullet&{\mathfrak r}&q},$}
\put(72,-9){$A$}
\end{picture}
\]
$x\in \{{\mathfrak r}\}\cup A$, and in the notation of swap (IV) we have ${\mathfrak m}=x$.
More precisely, \emph{suppose} one can find a set of labels
${\mathfrak r}=x-d,x-d+1,\ldots,x-1,x$ (for some $d\geq 0$) where $x-d,x-d+1,\ldots,x-1$
are labels in the edge above the box containing $x$ in $U$ and
\[{\mathcal N}_{{\rm col} \ 1, x-i}^U={\mathcal N}_{{\rm col} \ 1, x}^U\]
for $1\leq i\leq d$. Further \emph{suppose} if those labels are moved where $A$ is (and combined with labels already on that edge in $U$)
then $V'$ is good. In this case, take $d$ to be maximal
among all choices satisfying these conditions and define $A$ and thus $V'$ in this manner.

\begin{Subclaim}
If $d\geq 0$ exists then $V'$ is lattice.
\end{Subclaim}
\begin{proof}
$U$ is lattice (by the induction hypothesis) and
only two columns of $U$ change to construct $V'$.
Thus, if $V'$ is not
lattice, the failure of latticeness can be blamed on one of these two columns. It cannot be the
first column of the local picture of $V'$ since we moved labels rightward and thus
${\mathcal N}_{{\rm col}\ 1,t}^U={\mathcal N}_{{\rm col}\ 1,t}^{V'}$ for any label $t$.
If there is a problem in the second column, it would have to be that
${\mathcal N}_{{\rm col}\ 2, x-d}^{V'}>{\mathcal N}_{{\rm col}\ 2, x-d-1}^{V'}$, so
assume this holds.
Since
$U$ is lattice we have ${\mathcal N}_{{\rm col}\ 1, x-d-1}^U\geq{\mathcal N}_{{\rm col}\ 1, x-d}^U$.  In combination with our assumption, it must be that
${\mathcal N}_{{\rm col}\ 1, x-d-1}^U={\mathcal N}_{{\rm col}\ 1, x-d}^U$, and
$x-d-1$ appears in column 1 of $U$.  It must appear either in the edge
above $x$ or in the box above it.
We also note that $x-d-1$ cannot appear in column 2 of $U$, since
if it did, we would have $\mathcal N_{{\rm col}\ 3, x-d-1}^U<\mathcal N_{{\rm col}\ 3,x-d}^U$, contrary to the assumption that $U$ is lattice.

Suppose first that $x-d-1$ appears in the first column of $U$
in the box above $x$.  That is to say, using
our labelling of entries of $U$ defined above, that $z=x-d-1$.  Now consider
the value $y$.
Since
we have assumed that $V'$ is good, we must have $y<x-d$, and semistandardness
requires $y\geq z=x-d-1$.  So $y=x-d-1$, but that contradicts our argument
above that $x-d-1$ does not appear in column 2 of $U$.

Now suppose that $x-d-1$ appears on the edge above $x$.  Since we know
that $x-d-1$ does not appear in column 2, we could have chosen
$\mathfrak r=x-d-1$ rather than $\mathfrak r = x-d$, which contradicts
the fact that $d$ was chosen to be maximal.

We have found a contradiction based on our assumption that $V'$ was not
lattice, so it must be that $V'$ is lattice.
\end{proof}

\begin{Subclaim}
Suppose ${\widetilde V}'$ is good, lattice and ${\widetilde V}'\mapsto U$ is obtained by swap (IV). Then ${\widetilde V}'$ is unique (and hence ${\widetilde V}'=V'$ as just constructed above).
\end{Subclaim}
\begin{proof}
The only question is whether in our given construction of $V'$ we can instead use $0\leq d'<d$ in place
of $d$. That is, we construct ${\widetilde V}'$ by moving fewer labels right than we could have,
i.e., we move ${\mathfrak r}=x-d',x-d'+1,\ldots,x-1,x=z$. If we do this then note that $\widetilde V'$ is not lattice since
\[{\mathcal N}_{{\rm col}\ 2, x-d'}^{\widetilde V'}={\mathcal N}_{{\rm col}\ 1,x-d'}^U={\mathcal N}_{{\rm col}\ 1, x-d'-1}^U={\mathcal N}_{{\rm col}\ 2, x-d'-1}^{\widetilde V'}+1.\]
(The first equality holds since there is no $x-d'$ in column $2$ of $U$.) Hence we find
${\mathcal N}_{{\rm col}\ 2, x-d'}^{\widetilde V'}>{\mathcal N}_{{\rm col}\ 2, x-d'-1}^{\widetilde V'}$,
so $\widetilde V'$ is not lattice.
\end{proof}
% The following lemma should eventually be unnecesary, by definition. However I put it in for now
% just to simplify my current editing (03/23/12)

\begin{Subclaim}
If $V$ and $V'$ are good and lattice then they cannot both result in $U$, using swaps (III) and (IV) respectively.
\end{Subclaim}
\begin{proof}
If (III) could be
applied to $V$ to give $U$ then
\[\begin{picture}(120,35)
\put(20,15){$U=\tableau{{z}&y&w\\x&\bullet&q}$}
\put(70,-7){$x$}
%\put(130,15){or $U_2=\tableau{{z}&y&w\\x&\bullet&q}$}
%\put(194,-6){$x\!\!+\!\!1$}
\end{picture}\]
where the edge label $x$ is the least label on its edge. However, then
$V'$ is ruled out since we must have two $x$'s in the second column of
$V'$, a contradiction.\end{proof}

\begin{Subclaim}
One can actually reverse from $U$ using either (III) or (IV).
\end{Subclaim}
\begin{proof}
Let $\gamma$ be the smallest label on the edge directly below the
$\bullet$ in $U$. It satisfies $x\leq \gamma$ (since $U$ is good).
If $\gamma=x$ we saw (III) is applicable: $V\mapsto U$ where $V$ is good and lattice.
If $\gamma>x(>y)$ then since $x\leq q$ (since $U$ is really good) the construction of $V'$
can be achieved, and we saw $V'$ is good and lattice, as desired.
\end{proof}

\noindent
{\bf Case 2}: \emph{Suppose
\[\begin{picture}(100,35)
\put(20,15){$U=\tableau{d&e&f\\x&\bullet&t}$}
\put(70,13){$y$}
\end{picture}
\]
where $y$ is the largest label in its edge.}

\noindent
\emph{Subcase 2a: $x\leq y$:} Clearly a vertical swap (I) could not have produced $U$. If a
resuscitation swap (III) produced $U$ then $V$ looks locally like
\[\begin{picture}(100,20)
\put(20,5){$V=\tableau{\bullet&x}$}
\put(50,21){$x$}
\put(70,21){$y$}
%\put(50,0){$\beta$}
\end{picture}
\]
where semistandardness requires $y<x$. This contradicts the assumption of this subcase.
On the other hand, if a horizontal swap (IV) produced $U$ then
\[\begin{picture}(100,35)
\put(20,15){$V=\tableau{d&e&f\\\bullet&{\mathfrak r}&t}$}
\put(70,13){$y$}
\put(68,-6){$A$}
\end{picture}
\]
where $x\in\{{\mathfrak r}\}\cup A$, which by vertical semistandardness implies that
$x>y$, which is again a contradiction.

Finally, consider
\begin{equation}
\label{eqn:expansionmove}
\begin{picture}(100,35)
\put(20,15){$V=\tableau{d&e&f\\x&\bullet&t}$}
\put(70,-5){$y$}
\end{picture}
\end{equation}
where $y$ is the least label on its edge (the other labels being those on the same edge
of $U$.)
Clearly $V$ is good and lattice (since we assume $x\leq y$ and $U$ is good and lattice) and an expansion swap (II) produces $U$.

\noindent
\emph{Subcase 2b: $x>y$:} Clearly $U$ did not arise from a vertical swap (I). Next, suppose an expansion swap (II)
produced $U$. Then $V$ is of the form (\ref{eqn:expansionmove}), where
 $y$ is the least element on its edge. But $y<x$, so $V$ is not good.

A resuscitation swap (III) can produce $U$ if
\[
\begin{picture}(200,35)
\put(20,15){$V=\tableau{d&e&f\\\bullet&x&t}$}
\put(52,13){$x$}
\put(70,13){$y$}

\put(110,15){$\mapsto$}
\put(130,15){$U=\tableau{d&e&f\\x&\bullet&t}$}
\put(180,-7){$x$}
\put(180,13){$y$}
\end{picture}
\]
%We must also require $U$ to be of the form coming from $V$ using (III).
%HT SAYS: I THINK IT WOULD BE CLEARER TO GIVE THE EXPLICIT PICTURE FOR $U$.

Suppose the resuscitation swap (III) is not possible
starting with $V$. We need to construct a unique
\[
\begin{picture}(100,35)
\put(20,15){$V'=\tableau{d&e&f\\\bullet&x&t}$}
\put(73,13){$y$}
\put(72,-7){$Y$}
\end{picture}
\]
such that $V'$ is good and lattice, and $V'\mapsto U$ using (IV).
The arguments are exactly the same as in subcase 1b.

We have now completed our proof of Claim~\ref{claim:good}.
\excise{As in subcase 1b, we must argue that exactly one of $V$ and $V'$ results in $U$. (The arguments are
in fact essentially the same.) If (III) could be applied to $V$ could be applied to $V$ to give
$U$ then either $U$ looks like
\[\begin{picture}(240,35)
\put(20,15){$U_1=\tableau{{d}&e&f\\x&\bullet&t}$}
\put(74,-6){$x$}
\put(74,13){$y$}
\put(130,15){or $U_2=\tableau{d&e&f\\x&\bullet&t}$}
\put(194,-6){$x\!\!+\!\!1$}
\put(198,13){$y$}
\end{picture}
\]
If $U=U_1$ then $V'$ could not be a possibility since we would have a violation of vertical
semistandardness in the second column of $V'$. Therefore say $U=U_2$. Hence $x+1\in Y$ but
does not move left. The exact same argument as in subcase 1b shows that the only reason for this
is that (III)(ii) does not hold. However, this implies $V\to U$ did not happen after all.}
\qed

\begin{Claim}
\label{claim:nearbad}
In the process of reversing from $W$, if we arrive at a tableau $U=U^{(-i)}$ that is nearly bad, then the forward step $U\mapsto U^{\star}=U^{(-i+1)}$ was a horizontal swap.
\end{Claim}
\begin{proof}
By assumption, locally we have
\[U=\tableau{z&y&w\\x&\bullet&q\\{j}&k&{m}},\]
where $x>q$. We show that $U\mapsto U^{\star}$
could not be swaps (I), (II) and (III).

Suppose $U\mapsto U^{\star}$ is swap (I). Then $k\leq q$. But then $U^{\star}$ is not good since $x>k$ and $x$ and $k$ are adjacent in $U^{\star}$; this is a contradiction. Similarly, we could not have used swap (II). Finally, if swap (III) was used, then
$q={\mathfrak u}$ where ${\mathfrak u}$ is the largest label in the upper
edge of the box in $U$ with the $\bullet$. But $x>q=u$ means that, again,
 $U^*$ would not be good.
\end{proof}

Claim~\ref{claim:good} tells us how to reverse from $W$ until we arrive at a nearly bad tableau $U$. Claim~\ref{claim:nearbad} says that we can only arrive at a nearly bad tableau by (reversing) a horizontal swap (IV). The remaining claim below explains how to reverse from a nearly bad tableau:

\begin{Claim}
\label{claim:isnearlybad}
Suppose we are in the process (\ref{eqn:uniqueseq})
of reversing from $W$ and we
arrive at a nearly bad $U=U^{(-i)}$.
Then there is a good and lattice tableau $V$ such that $V\mapsto U$ is a swap (IV). If $V$ is nearly bad, the defect occurs in the same row as
the defect of $U$, but one square to the left.
\end{Claim}
\begin{proof}
By Claim~\ref{claim:nearbad} we may suppose
$U\mapsto U^{\star}=U^{(-i+1)}$, where the local pictures are
\[\begin{picture}(240,35)
\put(20,15){$U=\tableau{{z}&y&w\\x&\bullet&q}$}
\put(87,-8){$Y$}
\put(67,-8){$C$}
\put(49,-8){$B$}
\put(49,10){$A$}
\put(69,10){$T$}
\put(130,15){$U^{\star}=\tableau{{z}&y&w\\x&f&\bullet}$}
\put(200,-8){$Y'$}
\put(184,-8){$C$}
\put(164,-8){$B$}
\put(184,12){$W$}
\put(164,10){$A$}
\end{picture}
\]
and $x>q$ (since $U$ is nearly bad).

We need to show that we can take some of the labels of $A$ and move them right so as to construct
\[\begin{picture}(100,35)
\put(20,15){$V=\tableau{{z}&y&w\\\bullet&{\mathfrak r}&q}$}
\put(87,-8){$Y$}
\put(67,-8){$C'$}
\put(49,-8){$B$}
\put(49,10){$A'$}
\put(69,10){$T$}
\end{picture}
\]
where all the conditions on being good (but possibly nearly bad) are met, and
$V\mapsto U$ using (IV).

We have $x\leq f<\min C$ (since $U^{\star}$ is good). Also, $x>q$ and $U\mapsto U^{\star}$ occurs, so $x>q>\max T$.
 Hence $x$ can be placed into $C$'s edge and maintain vertical semistandardness in that column.
 Note that ${\mathfrak r}=x$ is not possible since then $V$ is bad.
  Let $A=\{a_k<a_{k-1}<\ldots<a_1\}$ where $a_1<x$ (by column
semistandardness).  We need to show there exists  $j\geq 1$ satisfying the following conditions:
\begin{itemize}
\item $a_{j},a_{j-1},\dots,a_1,x$ forms an interval,
\item $\mathcal N^U_{{\rm col}\ 1,a_j} = \mathcal N^U_{{\rm col}\ 1,x}$,
\item $a_j$ is strictly larger than the maximum entry of $T$ (or $y$, if $T$
is empty).
\end{itemize}
 Then choose $j$ to be maximal subject to those conditions. We want to establish that $a_j \leq q$ so that we can set $\mathfrak r=a_j$
and $C'=C\cup \{a_{j-1},\dots,a_1,x\}$, and have $V$ be good.

Now, since $U\mapsto U^{\star}$ using swap (IV) we know $q+1,q+2,\ldots,f-1,f\in Y$. Moreover,
by the prerequisite (IV)(ii) we have
${\mathcal N}_{{\rm col}\ 2, i}^{U^{\star}}={\mathcal N}_{{\rm col}\ 2, q}^{U^{\star}}$
for $q\leq i\leq f$. Using this, together with the fact that $q<x\leq f$,
and the fact that the first column of $U$ and $U^{\star}$ are
the same, we deduce that there exists an $x-1$ in column $1$ of $U$: otherwise we find that
${\mathcal N}_{{\rm col}\ 1,x-1}^{U^\star}<{\mathcal N}_{{\rm col}\ 1,x}^{U^\star}$ so that
$U^{\star}$ is not lattice (contradicting our induction hypothesis).
If $x-1\not\in A$ it
must be $z$. But then $y\geq x-1$ which
contradicts that $U\mapsto U^{\star}$ is possible. Hence $x-1\in A$.
Continuing this same reasoning implies $x-2,x-3,\ldots,q+1,q\in A$.
It then follows that $a_j\leq q$, so $V$ is good.

We now check that $V\mapsto U$.  The only concern is if $x+1\in C'$, so that $x+1$ might also move left when we apply the horizontal swap (IV), so that we do not arrive at $U$ after all.  However, if this were true then $\mathcal N^V_{{\rm col}\ 2,x}=
\mathcal N^V_{{\rm col}\ 2,x+1}$.  This would imply that
$\mathcal N^U_{{\rm col}\ 2,x}<\mathcal N^U_{{\rm col}\ 2,x+1}$, violating the
lattice property of $U$.

It remains to check that $V$ is lattice. Recall $U$ is lattice (by the induction hypothesis) and
$V$ and $U$ agree except in two columns. Since we are moving labels to the right from column $1$
of $U$ into column $2$, if $V$ is not lattice we have ${\mathcal N}_{{\rm col}\ 2,a_j}^V>{\mathcal N}_{{\rm col}\ 2, a_j -1}^V$.

In order for this to happen, we must have an $a_j-1$ in column $1$ of $U$.
Further, there must be no $a_j-1$ in column $2$ of $U$, since otherwise
$\mathcal N_{{\rm col}\ 2, a_j -1}^U >\mathcal N_{{\rm col}\ 2, a_j}^U$, and $U$
is not lattice, contrary to our assumption.

Hence, it must be true that ${\mathcal N}_{{\rm col}\ 1,a_j}^U={\mathcal N}_{{\rm col}\ 1,a_j-1}^U$.
Moreover, in fact $a_j-1\in A$: Otherwise in $U$, $z=a_j-1$. Since $y\neq a_{j} -1$, by $U$'s goodness, $y\geq a_j$
implying $V\mapsto U$ is impossible, and thus violating the definition of $a_j$.
Therefore we should also have moved
$a_{j}-1(=a_{j+1})$ in our construction of $V$. This contradicts the maximality of $j$.

Summarizing, $V$ is good, but possibly nearly bad: It might be that
${\mathfrak r}$ is
strictly smaller than the first numerical label to its left (if it exists). However, in this case, the near badness has moved one square left, as claimed.
\end{proof}

\noindent
\emph{Conclusion of the proof of the Theorem~\ref{thm:DequalsC}:}
First suppose we are considering the case (b) and
our initial tableau $U^{(0)}$ that we are reversing from is obtained from $W$ by pushing the label
in box $\bbb$ to its lower edge.
Then $U^{(0)}$ is really good and lattice. So we are in the situation
of Claim~\ref{claim:good} and can take a first step in the
reversal process (\ref{eqn:uniqueseq}). If this reversal
results after some steps in a nearly bad tableau, then we can utilize Claim~\ref{claim:nearbad}
and Claim~\ref{claim:isnearlybad}. At each step we obtain a good tableau with
 strictly fewer labels northwest of the $\bullet$. Thus, by induction, we eventually
arrive at the situation that the $\bullet$ has no labels northwest of it. This happens
when $\bullet$ arrives at an outer corner of $\lambda$. Call the final
tableau $T$ of shape $\lambda^+$. Then $T$ is good (thus semistandard) and lattice.
Moreover, the final position of $\bullet$ and $T$ itself was uniquely determined from $U^{(0)}$. This completes the proof for (b). The argument for (a) is the same, except we start with
$U^{(0)}=W$.\qed

\subsection{Proof of Proposition~\ref{prop:DequalsCspecial}}

We now show that $C_{\lambda,\mu}^{\lambda}=D_{\lambda,\mu}^{\lambda}$, a fact we needed in the above proof of
Theorem~\ref{thm:DequalsC}.

For $\lambda \subseteq \Lambda=k\times (n-k)$, the
{\bf Grassmannian permutation} associated to $\lambda$
is the permutation $\pi(\lambda)\in S_n$ uniquely defined by $\pi(\lambda)_i=i+\lambda_{k-i+1}$
for $1\leq i\leq k$ and which has at most one descent, which (if it exists) appears at position $k$.

Let $w',v'\in S_n$ be the Grassmannian permutations for the conjugate shapes $\lambda',\mu'\subseteq (n-k)\times k$.
The following
identity relates $C_{\lambda,\mu}^{\lambda}$ to the localization at $e_{\mu}$ of the class $\sigma_{\lambda}$,
as expressed in terms of a specialization of the double Schubert polynomial. It is well known to experts;
it can be proved (in the conventions we use) by, e.g.,
combining \cite[Lemma~4]{Knutson.Tao} and \cite[Theorem~4.5]{Woo.Yong:KLideals}:
\[C_{\lambda,\mu}^{\lambda}(Gr_{k}({\mathbb C}^n))=\overline{{\mathfrak S}_{v'}(t_{w'(1)},\ldots,t_{w'(n)};t_1,\ldots,t_n)}.\]

Here $\overline{p(t_1,\ldots,t_n)}$ is the polynomial obtained from $p(t_1,\ldots,t_n)$ under the substitution $t_j\mapsto t_{n-j+1}$.
We refer the reader to \cite{Manivel} for background about Schubert polynomials; however, we will only use a subset of the theory,
which we describe now.

Since $v'$ is Grassmannian, we have
\[{\mathfrak S}_{v'}(X;Y)=\sum_{T}{\tt SSYTwt}(T)\]
where the sum is over all (ordinary) semistandard Young tableau $T$ of shape $\mu'$ with entries bounded above by $n-k$.
Here ${\tt SSYTwt}(T)=\prod_{\bbb\in \mu'}(x_{{\rm val}(\bbb)}-y_{{\rm val}(\bbb)+j(\bbb)})$ where $j(\bbb)={\rm col}(\bbb)-{\rm row}(\bbb)$.
This formula is well-known (see, e.g., a more general form in \cite[Theorem~5.8]{KMY}).

The Schubert polynomial $\mathfrak S_{v'}$ for a Grassmannian permutation
$v'$
can also be identified as the
{\bf factorial Schur function} $s_{\mu'}$ (cf.~\cite{Louck,Macdonald,Goulden}):
One has (see, e.g., \cite[Section~2]{Kreiman:Eq}),
after (re)conjugating the shapes, that if we take $\lambda,\mu\subseteq\Lambda$ then
$s_{\lambda}\cdot s_{\mu} =\sum_{\nu\subseteq \Lambda}C_{\lambda,\mu}^{\nu}(t_j\mapsto -y_j)s_{\nu}$.
We will not need this identification.
% I undid my incorrect "correction".

Let ${\tt SSYTeqwt}(T)$ be the result of the substitution $x_j\mapsto t_{w'(j)}, y_j\mapsto t_j$.
Define ${\mathcal A}$ to be the set of semistandard and lattice tableaux $T$ of shape $\lambda/\lambda$ and content $\mu$ such
that ${\tt apwt}(T)\neq 0$. Define ${\mathcal B}$ to be the set of semistandard tableaux $U$ of shape $\mu'$ where
${\tt SSYTeqwt}(U)\neq 0$.

It remains to prove the following:

\begin{Claim}
There is a weight-preserving bijection $\phi:{\mathcal A}\to {\mathcal B}$
 where if $T\in \mathcal A$ then ${\tt apwt}(T)=\overline{{\tt SSYTeqwt}(\phi(T))}$.
\end{Claim}
\begin{proof}
  Define $\phi$ as follows. Label the columns of $\Lambda=k\times (n-k)$
by $(n-k),(n-k)-1,\ldots,3,2,1$ \emph{from left to right}.
Given $T$, let ${\tt col}(T)$ be the word $c_1 c_2 \cdots c_{|\mu|}$ obtained by
recording the column indices of the $1$'s (from left to right), $2$'s (from left to
right) etc. Now let $\phi(T)$
be obtained by placing this word into the boxes of shape $\mu'$ from bottom to top along columns,
and from left to right (noting there are $\mu_i$ labels $i$ in $T$ for each $i$).
We have a candidate inverse map $\phi^{-1}:{\mathcal B}\to {\mathcal A}$ obtained by reading $U\in {\mathcal B}$ in the same way
and placing edge labels on the bottom edge of $\lambda/\lambda$: the placement of
the $i$'s is determined by the labels in column $i$ of $U$.

\begin{Example}
Let $n=7$, $k=3$, $\lambda=(4,2,1)$ and $\mu=(4,2)$. Then $T$, together with the column labels
${\sf 1},\ldots,{\sf 4}$ and $\phi(T)$ are depicted below:
\[\begin{picture}(240,63)
\put(50,58){${\sf 4}$}
\put(70,58){${\sf 3}$}
\put(88,58){${\sf 2}$}
\put(108,58){${\sf 1}$}
\put(20,35){$T=\tableau{{\ }&{ \ }&{\ }&{ \ }\\{\ }&{\ }\\{\ }}$}
\put(45,-7){$1,2$}
\put(65,13){$1,2$}
\put(88,31){$1$}
\put(108,31){$1$}
\put(140,28){$\mapsto$}
\put(160,28){$\phi(T)=$}
\put(200,38){$\ktableau{{1}&{3}\\{2}&{4}\\{3}\\{4}}$}
\end{picture}
\]
Here we had ${\tt col}(T)=4 3 2 1 4 3$.

We compute
\[{\tt apwt}(T)=(t_1-t_7)(t_3-t_7)(t_5-t_7)(t_6-t_7)(t_1-t_4)(t_3-t_4),\]
where the first four factors correspond to the labels $1$ of $T$ from left to right and
the last two factors correspond to the labels $2$ of $T$ from left to right. Now,
\[{\tt SSYTwt}(\phi(T))=(x_4-y_1)(x_3-y_1)(x_2-y_1)(x_1-y_1)(x_4-y_4)(x_3-y_4),\]
where the factors correspond to the entries of $\phi(T)$ as read up columns from left
to right (i.e., consistent with the order of factors of ${\tt apwt}(T)$ above).

Since $\lambda'=(3,2,1,1)$ and $\mu'=(2,2,1,1)$ we have $w'=2357146$ and $v'=2356147$
(one line notation). So substituting, we get
\[{\tt SSYTeqwt}(\phi(T))=(t_7-t_1)(t_5-t_1)(t_3-t_1)(t_2-t_1)(t_7-t_4)(t_5-t_4).\]
Finally, the reader can check ${\overline{{\tt SSYTeqwt}(T)}}={\tt apwt}(T)$, in agreement
with the Claim.\qed
\end{Example}
%HT SAYS: AN EXAMPLE HERE WOULD HELP. (SUGGESTION?)
%TRANSPOSED QUESTION

($\phi^{-1}$ is well-defined and is weight-preserving): Let $U\in {\mathcal B}$.
Since $\phi^{-1}(U)$ is of shape $\lambda/\lambda$, it is vacuously standard.
The fact that $U$ is semistandard easily implies that $\phi^{-1}(U)$ is lattice.

We check that the weight assigned to a label $\ell$ in box ${\sf b}$ and column $c={\tt col}(\bbb)$ of $U$ is the same as the ${\tt apfactor}$ assigned to the corresponding label $c$ in $\phi^{-1}(U)$. The label $\ell$ gets assigned the weight
${\tt SSYTeqfactor}=t_{\lambda_{(n-k)-\ell+1}'+\ell}-t_{\ell+j(\bbb)}$. Hence we must show the equality of these two quantities:
\begin{eqnarray*}
\overline{{\tt SSYTeqfactor}(\ell)}&=&t_{n-(\lambda_{(n-k)-\ell+1}'+\ell)+1}-t_{n-(\ell+j({\sf b}))+1}, \textrm{ and}\\
{\tt apfactor}(c)&=&t_{{\tt Man}(\x)}-t_{{\tt Man}(\x)+r-c+1+
\mbox{$\#$ of $c$'s strictly to the right of $\x$}},\end{eqnarray*}
where here $\x$ is the bottom edge of $\lambda$ in column $\ell$
from the right edge of $\Lambda$ and $r=\lambda_{(n-k)-\ell+1}'$.

Now, counting the number of columns and rows which separate $\x$ from the
bottom-left corner of $\Lambda$, we have
\[{\tt Man}(\x)=((n-k)-\ell)+(k-\lambda_{(n-k)-\ell+1}'+1)=n-(\lambda_{(n-k)-\ell+1}'+\ell)+1.\]
Thus, the first term of ${\overline{{\tt SSYTeqfactor}(\ell)}}$ and ${\tt apfactor}(c)$
agree. To compare the second terms note that
\begin{multline}\nonumber
{\tt Man}(\x)+r-c+1+ \mbox{$\#$ of $c$'s strictly to the right of $\x$}=\\
[n-(\lambda_{(n-k)-\ell+1}'+\ell)+1]+\lambda_{(n-k)-\ell+1}'-c+1+\mbox{$\#$ of $c$'s strictly to the right of $\x$}\\
=n-\ell+1-c+1+\mbox{$\#$ of $c$'s strictly to the right of $\x$}
\end{multline}
Hence it suffices to show
\[-j({\sf b})=-c+1+\mbox{$\#$ of $c$'s strictly to the right of $\x$},\]
or equivalently,
\[{\rm row}({\sf b})-1=\mbox{$\#$ of $c$'s strictly to the right of $\x$}.\]
However, this final equality is clear by the definition of $\phi^{-1}$.

Thus $0\neq \overline{{\tt SSYTeqwt}(U)}={\tt apwt}(\phi^{-1}(U))$ and we are done.

($\phi$ is well-defined and weight-preserving): Let $T\in \mathcal A$.
By construction, $\phi(T)$ is strictly increasing along columns.
%Suppose $W\in {\tt EqSYT}(\alpha/\beta)$ for a skew shape $\alpha/\beta$. Then if in $W$ %the label $i+1$
%is south and weakly west of $i$, then the same holds for the (column) rectification of %$W$, denoted ${\tt colrect}(W)$.
%
%Thus, since ${\tt colrect}(T)=S_{\mu'}$ it follows that each set of labels
%\[\{1,2,\ldots,\mu_1'\}, \ \{\mu_1'+1,\ldots,\mu_1'+\mu_2'\},\ldots,\]
%that is, each $i$ is weakly southwest of $i+1$. Hence $\phi(T)$ is clearly column %semistandard.

Now suppose $\phi(T)$ is not weakly increasing along rows. Thus there is a violation between columns $c+1$ and $c$. We may suppose $c+1$ is
the leftmost column of $\Lambda$, recalling the reverse labelling of columns;
the general argument is similar. Now suppose the violation occurs
$M$ rows from the top. Hence in $T$, the $M$-th label $1$ (counting from the
right) is in a column strictly to the left of the label
$M$-th label $2$. Then it must be true that $T$ is not lattice.

Hence $\phi(T)$ is a semistandard tableau of shape $\mu'$. The same computations showing $\phi^{-1}$
is weight preserving shows $0\neq {\tt apwt}(T)=\overline{{\tt SSYTeqwt}(\phi(T))}$ and so the desired
conclusions hold.
\end{proof}
\subsection{Proof of Theorem~\ref{thm:main}}

Let ${\mathcal C}$ be the set of lattice semistandard tableaux $S$ of shape $\nu/\lambda$ whose content
is $\mu$ and ${\tt apwt}(S)\neq 0$. Also, let ${\mathcal D}$ be the set of tableaux from
Theorem~\ref{thm:main}. Define a map $\Phi:{\mathcal C}\to {\mathcal D}$ as follows: given $S\in {\mathcal C}$
relabel the $\mu_1$ labels $1$ that appear by $1,2,\ldots,\mu_1$, from left to right; then relabel the $\mu_2$
(original) labels $2$ by $\mu_1+1,\mu_1+2,\ldots,\mu_1+\mu_2$, etc. This map is clearly reversible.
Theorem~\ref{thm:main}  follows from:

% START OF NEW ARGUMENT
\begin{Proposition}
$\Phi:{\mathcal C}\to {\mathcal D}$ is a weight preserving bijection:  ${\tt apwt}(S)={\tt wt}(\Phi(S))$.
\end{Proposition}
\begin{proof}
($\Phi$ is well-defined): Since $S\in {\mathcal C}$ is semistandard, clearly $T=\Phi(S)$ is standard.

Let $T_{\mu}[i]$ be the set of labels in row $i$ of $T_{\mu}$. By construction, the labels of
$T_{\mu}[i]$ form a horizontal strip in $T$. The following is an easy induction using the
definition of ${\tt Ejdt}$:

\begin{Claim}
\label{claim:horizstripCD}
The labels of $T_{\mu}[i]$ form a horizontal strip in each tableau arising in the process of column rectifying $T$.
\end{Claim}

Translating the assumption that $S$ is lattice, for any column $c$ of $T$, the number of labels from $T_{\mu}[i]$ appearing in columns weakly to the right of column $c$ weakly exceeds the number from $T_{\mu}[i+1]$ in the same region, for any $i\geq 1$. Mildly abusing terminology, we say that $T$ is also lattice.

\begin{Claim}
\label{claim:latticeCD}
Each tableau appearing in the column rectification of $T$ is lattice.
\end{Claim}
\begin{proof}
Suppose that in the process of column rectification we arrive at a tableau $U$ (which may have a $\bullet$ in the middle of it) which is lattice and the next swap $U\mapsto U'$ breaks latticeness.
Then this swap must locally look like
$U=\ktableau{{a}&{b}\\{\bullet}&{c}\\{d}&{e}}\mapsto
\ktableau{{a}&{b}\\{c}&{\bullet}\\{d}&{e}}=U'$
where $c\in T_{\mu}[i]$ moving left causes more labels of $T_{\mu}[i+1]$ than of $T_{\mu}[i]$ to appear weakly right of column $2$ of $U'$. So there must be a label $\ell$ of $T_{\mu}[i+1]$ in
column $2$ of $U$ (and of $U'$), since otherwise $U$ is not lattice, a contradiction.

Suppose $e$ does not exist. Then since $U$ is standard, $\ell$ cannot
exist, a contradiction. Hence we assume $e$ and thus $d$ exists.
By standardness of $U$ and Claim~\ref{claim:horizstripCD}, $e(=\ell)\in T_{\mu}[i+1]$.
Notice that no label of column $1$ of $U$ can be in $T_{\mu}[i]$ since
we would contradict Claim~\ref{claim:horizstripCD} (applied to $U'$). Now $d>c$ (since otherwise the swap
would not have been used). So by standardness and Claim~\ref{claim:horizstripCD} (applied to $U$), $d\in T_{\mu}[i+1]$. But then $U$ was not lattice in column $1$ to
begin with. This is our final contradiction.
\excise{Suppose $e$ does not exist. Then since $U$ is standard $\ell$ cannot exist, a contradiction. Hence we assume $e$, and thus
$d$, exists. Now $d>c$ (since otherwise the swap would not have been used), but $d\not\in T_{\mu}[i]$ since that would violate Claim~\ref{claim:horizstripCD}. So $d\in T_{\mu}[j]$
for some $j\geq i+1$. If $j=i+1$ no label in column $1$ of $U$ is in $T_{\mu}[i]$ (by standardness of $U$ combined with Claim~\ref{claim:horizstripCD}). This is a contradiction since we conclude $U$ was not lattice in the first column.  So $j>i+1$ but then by standardness of $U$,
$e\in T_{\mu}[k]$ for $k>i+1$, so again $\ell$ cannot exist, our final contradiction.}
\end{proof}

Write $T^{(k)}$ for the tableau that consists of
the $k$ rightmost columns of the column rectification of the $k$
rightmost columns of $T$.

\begin{Claim}
\label{claim:comparison}
The $i$-th row of $T^{(k)}$ is a consecutive sequence of integers from $T_{\mu}[i]$, ending with $\mu_1+\dots+\mu_i(=\max T_{\mu}[i])$.
\end{Claim}
\begin{proof}
The argument is by induction on $k\geq 0$. The base case $k=0$ is trivial. Suppose after rectifying the $k-1$ rightmost columns of $T$,
$T^{(k-1)}$ has the claimed form. Now we are rectifying column $k$ (from the right). Suppose we are ${\tt Ejdt}$ sliding into a square ${\sf x}$ in row $R$ and the slide
${\tt Ejdt}_{\sf x}$ is a horizontal one (i.e., a label moves left). Observe that in this case,
the $\bullet$ must only move right in the same row until the slide completes: otherwise, by the form of $T^{(k-1)}$, it must be that the rows $R$ and $R+1$ of $T^{(k)}$ are of the same length, and the
rightmost label of row $R+1$ moves up into row $R$; however this contradicts Claim~\ref{claim:latticeCD}.

Suppose the labels in the column we are presently rectifying are $\ell_1<\ell_2<\ldots <\ell_t$.
Now $\ell_m\in T_{\mu}[i_m]$ where $i_1<i_2<\ldots<i_t$. By the form of $T^{(k-1)}$, it is easy to see $\ell_m$ completes at row $i_m$. Now, by Claim~\ref{claim:horizstripCD} it follows that $\ell_m$ is the largest label of $T_{\mu}[i_m]$ that does not appear in $T^{(k-1)}$. This completes the
induction step.
\end{proof}

        Claim~\ref{claim:comparison} immediately shows ${\tt Erect}(T)=T_{\mu}$, as desired.

($\Phi^{-1}$ is well-defined): Let $T\in {\mathcal D}$. Let $S=\Phi^{-1}(T)$;
proving well definedness means we need to show $S$ is semistandard, lattice and ${\tt apwt}(S)\neq 0$ (the content of of $S$ being $\mu$ is by
construction).

\begin{Claim}
\label{claim:horizontal}
The labels $T_{\mu}[i]$ form a horizontal strip in $T$, as well as in each tableau $T'$ in the
column rectification of $T$.
\end{Claim}
\begin{proof}
Suppose $j$ and $j+1$ appear in the same row of $T_\mu$. Then we claim that $j+1$ is
strictly east (and, by standardness of $T$, thus weakly north) of $j$ in $T$ (respectively, $T'$). Otherwise,
if this is false, it remains false after each ${\tt Ejdt}$ step. This implies ${\tt Erect}(T)\neq T_{\mu}$, a contradiction.
\end{proof}

Given Claim~\ref{claim:horizontal}, the semistandardness of $S$ is clear.

Next we argue that $S$ is lattice. Otherwise, there is a column $c$ and label $i$ such that
${\mathcal N}_{{\rm col}\ c, i+1}^S>{\mathcal N}_{{\rm col}\ c, i}^S$.
We may assume $c$ is rightmost with this property. Hence $T$ is not lattice.

\begin{Claim}
\label{claim:everyswap}
Assuming (for the sake of contradiction) that $T$ is not lattice, it follows
that after every swap in the process that column rectifies $T$ to ${T}_{\mu}$, the resulting tableau is also not lattice.
\end{Claim}
\begin{proof}
Without loss of generality, it suffices
to argue about the first swap applied to $T$. If the result $T^{\circ}$ is lattice then there
is a label $\ell\in {T}_{\mu}[i+1]$ in column $c$ of $T$ that moved to the column $c-1$. Locally, the swap looks like
$\ktableau{{a}&{b}\\{\bullet}&{\ell}} \to \ktableau {a&b\\ \ell&\bullet}$.
By Claim~\ref{claim:horizontal}, the labels of ${T}_{\mu}[i+1]$ form a horizontal strip in $T$.
Hence $a,b\not\in {T}_{\mu}[i+1]$. Also, no label in column $c$ is in ${T}_{\mu}[i]$
since otherwise there is a violation of latticeness strictly to the right of column $c$ that is not fixed by this swap. Now, some label $m$ in column $c-1$ is in ${T}_{\mu}[i]$ (since we have fixed non-latticeness by the swap). This $m$ cannot be below the $\bullet$ since $\ell>m$ so $m$ would move into the $\bullet$
instead of $\ell$.  Hence $a=m$. Now what about $b$? We have excluded the possibility that $b\in {T}_{\mu}[i]\cup
{T}_{\mu}[i+1]$. However, by standardness of $T$, there are no other possibilities for $b$. This is
a contradiction and $T^{\circ}$ is not lattice.
\end{proof}

Thus, by Claim~\ref{claim:everyswap}, $T_{\mu}$ is not lattice, a contradiction. Hence $S$ is lattice.

Finally, in the weight preservation argument below, we see ${\tt apwt}(S)={\tt wt}(T)$. Thus we have ${\tt apwt}(S)\neq 0$ since by construction ${\tt wt}(T)\neq 0$.

($\Phi$ and $\Phi^{-1}$ are weight preserving): Suppose $T\in {\mathcal D}$
and we consider a label $\ell$ in that column which finishes in row $i$.
Claim~\ref{claim:comparison} (and its proof) shows that the labels to the right (and in the same row) of $\ell$ (once it completed rectifying in its column) are precisely those to its right in $T_\mu$, and moreover than any edge label rises exactly to its row in $T_{\mu}$
(although it may move left in that row in subsequent column rectifications). Hence by the definition of ${\tt apfactor}$, if $\ell'$ is the corresponding label in $S=\Phi^{-1}(T)$ then ${\tt factor}(\ell)={\tt apfactor}(\ell')$. So
${\tt wt}(T)={\tt apfactor}(\Phi^{-1}(T))$. Thus $\Phi^{-1}$ is weight-preserving. Reversing the
argument shows $\Phi$ is weight preserving.
\end{proof}

\section{Conjectural extension to equivariant $K$-theory}

The ring $K_{T}(Gr(k,{\mathbb C}^n))$ has a $K_{T}({\rm pt})$-basis of
equivariant $K$-theory classes
$\sigma_{\lambda}^K$ indexed by $\lambda\subseteq \Lambda$. Here
$K_T({\rm pt}):={\mathbb Z}[t_1^{\pm 1},t_2^{\pm 1},\ldots, t_n^{\pm 1}]$
is the Laurent polynomial ring in $t_1,\ldots,t_n$. Consequently, the {\bf equivariant
$K$-theory Schubert structure coefficients} are defined by the expansion
\begin{equation}
\label{eqn:KTexpansion}
\sigma_{\lambda}^K \cdot \sigma_{\mu}^K=\sum_{\nu}K_{\lambda,\mu}^{\nu}\sigma_{\nu}^K,
\end{equation}
where
$K_{\lambda,\mu}^{\nu}\in {\mathbb Z}[t_1^{\pm 1},t_2^{\pm 1},\ldots, t_n^{\pm 1}]$.

Earlier, a puzzle conjecture for these Laurent polynomials was given by
A.~Knutson-R.~Vakil and reported in \cite{Coskun.Vakil}. One aspect of their
conjecture is that it does not specialize to $K$-theory puzzle rules
(compare Sections 3 and 5 of \cite{Coskun.Vakil}
and see specifically the remarks of the fourth paragraph of the latter section). In contrast,
our conjecture transparently recovers the \emph{jeu de taquin} rules for $K$-theory, $T$-equivariant cohomology and ordinary cohomology, by ``turning off'' parts of our construction.

Recently, A.~Knutson \cite{Knutson:positroid} obtained a puzzle rule for an equivariant $K$-theory problem different than the one considered here (or in the Knutson-Vakil puzzle conjecture).

\subsection{Statement of the equivariant $K$-theory rule}
To state our conjectural generalization of Theorem~\ref{thm:main},
we need to broaden the class of equivariant tableaux. The ideas
contained below also generalize the
notions concerning \emph{increasing tableau} that we gave in our
earlier paper \cite{Thomas.Yong:V}, where a
\emph{jeu de taquin} rule for $K$-theory of Grassmannians was proved.

An {\bf equivariant increasing tableau} is an equivariant filling of
$\nu/\lambda$ by the labels $1,2,\ldots,\ell$ such that each label in a box
is:
\begin{itemize}
\item strictly smaller than the label in the box immediately to its right;
\item strictly smaller than the label in its southern edge, and the label in the
box immediately below it; and
\item strictly larger than the label in the northern edge.
\end{itemize}
Moreover, any subset of the boxes of $\nu/\lambda$
may be marked by a ``$\star$'', subject to:
\begin{itemize}
\item if the labels $i$ and $i+1$ appear as box labels in the same row of
$T$, then only the box containing $i+1$ may be marked by a ``$\star$''.
\end{itemize}
Let ${\tt EqINC}(\nu/\lambda,\ell)$ denote
the set of all equivariant increasing tableaux.

\begin{Example}
If $\nu/\lambda=(3,2)/(2)$ and $\ell=3$ the first two tableaux below are in
${\tt EqINC}(\nu/\lambda,\ell)$ while the third is not:

$ \ \ \ \ \ \ \ \ \ \ \ \ \ \ \ \ \ \ \ \ \ \ \ \ \ \ \ \ \ \ \ \ \ \ \ \ \ \  \ \ \ \
\begin{picture}(250,50)
\put(0,30){$\tableau{{\ }&{\ }&{1 }\\{1{\star} }&{3\star }}$}
\put(5,6){$2$}
\put(70,30){$\tableau{{\ }&{\ }&{2 }\\{1}&{2\star }}$}
\put(77,6){$2$}

\put(140,30){$\tableau{{\ }&{\ }&{2 }\\{1{\star} }&{2\star }}$}
\put(145,6){$2$}
\end{picture}
$\qed
\end{Example}

We also need an extension of the algorithms ${\tt Ejdt}$ and
${\tt Erect}$ defined in Section~1.

A {\bf short ribbon} $R$ is a connected skew shape that does not contain a
$2\times 2$ subshape and where each row and column contains at most two
boxes. An {\bf alternating ribbon} is a filling of $R$ by two symbols, say
$\alpha$ and $\beta$ such that
\begin{itemize}
\item adjacent boxes are filled differently;
\item all edges except the (unique) southmost edge are empty; and
\item if the southmost edge is filled, it is filled with a different symbol
than the symbol the in box above it.
\end{itemize}

\begin{Example}
The two types of alternating ribbons are of the form:
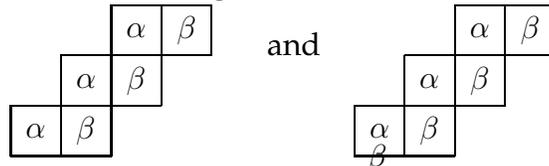
\[\begin{picture}(200,50)
\put(0,35){$\tableau{&&{\alpha}&{\beta}\\&{\alpha}&{\beta}\\{\alpha}&{\beta}}$ \ \ \ \ \ \ and }

\put(130,35){$\tableau{&&{\alpha}&{\beta}\\&{\alpha}&{\beta}\\{\alpha}&{\beta}}$}
\put(135,-7){$\beta$}
\end{picture}
\]
(where in the tableau on the right, the edge label $\beta$ is the smallest label
on that edge).
\end{Example}

We define ${\tt switch}(R)$ to be the alternating ribbon of the same shape
but where each box is instead filled with the other symbol. If the southmost
edge was filled by one of these symbols, that symbol is deleted.
If $R$ is a ribbon consisting of a single box with only one
symbol used, then ${\tt switch}$ does nothing to it. We also define
{\tt switch} to act on a skew shape consisting of multiple connected
components, each of which is a alternating ribbon, by acting on each
separately.

\begin{Example}
Applying {\tt switch} to either of the alternating ribbons above gives
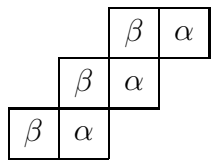
\[\tableau{&&{\beta}&{\alpha}\\&{\beta}&{\alpha}\\{\beta}&{\alpha}}\]
\end{Example}

Given $T\in {\tt KEqInc}(\nu/\lambda,\ell)$, consider
an {\bf inner corner} $\x\in\lambda$ which we label with a $\bullet$.
Erase all $\star$'s appearing in $T$. Consider
the alternating ribbon made of $\bullet$ and $1$. (It is allowed
for the southmost edge of $R_1$ in $T$ consists of the label $1$ and other
labels as well.) Apply {\tt switch} to $R_1$. Now let $R_2$ be the union of
ribbons consisting of $\bullet$ and $2$, and proceed as before. Repeat this
process until the $\bullet$'s have been switched past all the numerical
labels in $T$; the final placement of these labels gives
${\tt KEjdt}_{\x}(T)$. Finally, define ${\tt KErect}(T)$ by
successively applying ${\tt KEjdt}$ in the column rectification order.

\begin{Example}
\label{exa:aKEqrect}
Erasing the $\star$'s in
\[\begin{picture}(200,50)
\put(0,30){$T=\tableau{{\ }&{\ }&{2}\\{1\star}&{4}}$ \ gives \
$\tableau{{\ }&{\ }&{2}\\{1}&{4}}$}
\put(31,7){$3$}
\put(50,27){$1$}
\put(128,7){$3$}
\put(146,27){$1$}
\end{picture}\]
There is nothing to do to rectify the third column. Rectifying the second
column is achieved in one step:
\[\begin{picture}(240,45)
\put(40,25){$\tableau{{\ }&{\bullet}&{2}\\{1}&{4}}\mapsto
\tableau{{\ }&{1 }&{2}\\{1}&{4}}$}
\put(46,2){$3$}
\put(65,22){$1$}
\put(122,2){$3$}
\end{picture}\]
while rectifying the first column demands three steps:
\[\begin{picture}(300,50)
\put(40,30){$\tableau{{\bullet }&{1}&{2}\\{1}&{4}}\mapsto
\tableau{{1 }&{\bullet }&{2}\\{\bullet}&{4}}\mapsto
\tableau{{1 }&{2 }&{\bullet}\\{\bullet}&{4}}\mapsto
\tableau{{1 }&{2 }&{\bullet}\\{3}&{4}}
$}
\put(46,7){$3$}
\put(122,7){$3$}
\put(198,7){$3$}
\end{picture}\]
which gives the final tableau $T_{(2,2)}$.\qed
\end{Example}

While the definition of ${\tt KErect}$ above does not depend on the markings of
boxes of $T$ by $\star$, these markings play a role in
our modification of the equivariant weight ${\tt wt}(T)$
defined in Section~1.3. We say that a label $\mathfrak s \in T$ is a {\bf special label}
if it is either
\begin{itemize}
\item an edge label; or
\item lies in a box that has been marked by a $\star$.
\end{itemize}
To each special label ${\mathfrak s}$ we associate a Laurent binomial ${\tt factor}_{K}({\mathfrak s})$: given a
box $\x$ define a weight ${\hat \beta}(\x)=t_m/t_{m+1}$ where $m$ is the ``Manhattan distance'' as defined in Section~1. Note that at most one of the labels
``${\mathfrak s}$'' or ``${\mathfrak s}\star$'' can appear in a column. Moreover, each step of the
rectification moves an ${\mathfrak s}$ at most one step north
(and it remains in the same column). Therefore one can precisely say a
special label ${\mathfrak s}$ {\bf passes}
through a box $\x$ if it occupies it during the $K$-equivariant
rectification of the column that ${\mathfrak s}$ initially occupies \emph{and}
if ${\mathfrak s}$ did not initially begin in $\x$.
(This notion of ``pass'' reduces to our
original notion in Section~1.4 if ${\mathfrak s}$ is an edge label.)
Now, let $\x_1,\ldots,\x_s$ be the boxes passed through by ${\mathfrak s}$ and
$\y_1,\ldots,\y_t$ be the numerically labelled boxes in the same
row as $\x_s$ and strictly to its right. Set
\[{\tt factor}_{K}({\mathfrak s})=1-\prod_{i=1}^{s} {\hat \beta}(\x_i)\prod_{j=1}^{t}
{\hat\beta}(\y_j).\]
We apply the convention that if any special label ${\mathfrak s}$ does not move during
the rectification of the column that it initially sits in, then
${\tt factor}_{K}({\mathfrak s})=0$.
Now set
\[{\tt wt}_{K}(T)=\prod_{{\mathfrak s}} {\tt factor}_{K}({\mathfrak s}),\]
where the product is over all special labels ${\mathfrak s}$.

\begin{Example}
Assume that in Example~\ref{exa:aKEqrect}, we are working in
$Gr(2,{\mathbb C}^5)$. There are three special labels:
\begin{itemize}
\item The edge label ``$1$'' in the second column gives
${\tt factor}(1)=1-\frac{t_3}{t_4}\cdot\frac{t_4}{t_5}=1-\frac{t_3}{t_5}$
since it passes through one box during the rectification of column $2$, and
ends in a row with a single labelled box $\ktableau{{2}}$ to its right.
\item The marked label ``$1\star$'' gives
${\tt factor}(1\star)=1-\frac{t_2}{t_3}\cdot \frac{t_3}{t_4}=1-\frac{t_2}{t_4}$
since it passes through one box and has one box $\ktableau{{2}}$ to its right
after rectifying column $1$.
\item The edge label ``$3$'' gives
${\tt factor}(3)=1-\frac{t_1}{t_2}\cdot\frac{t_2}{t_3}=1-\frac{t_1}{t_3}$
since it passes through one box and has one box $\ktableau{{4}}$ to its right
when the rectification of column $1$ is complete.
\end{itemize}
Hence
${\tt wt}_K(T)=\left(1-\frac{t_3}{t_5}\right)\left(1-\frac{t_2}{t_4}\right)
\left(1-\frac{t_1}{t_3}\right)$.\qed
\end{Example}

Lastly, given $T$ we define
\begin{eqnarray}
\label{eqn:sign}
{\tt sgn}(T)& = & (-1)^{\#\mbox{\small $\star$'s in $T$}
\  +  \ \#\mbox{\small labels in $T$}\  -  \ |\mu|}\\ \nonumber
& = & (-1)^{\#\mbox{\small $\star$'s in $T$}\  +  \ \#\mbox{\small edge labels in $T$}\
+ \ |\nu|-|\lambda|- |\mu|}
\end{eqnarray}

\begin{Example}
Continuing Example~\ref{exa:aKEqrect} we have
${\tt sgn}(T)=(-1)^{1+2+5-2-4}=(-1)^2=1$.
\end{Example}

\begin{Conjecture}
\label{conj:mainconj}
The equivariant $K$-theory Schubert structure coefficient is
\[K_{\lambda,\mu}^{\nu}=\sum_{T} {\tt sgn}(T) \cdot {\tt wt}_{K}(T)\]
where the sum is over all $T\in {\tt EqINC}(\nu/\lambda,|\mu|)$
such that ${\tt KErect}(T)=T_{\mu}$.
\end{Conjecture}

\noindent
Conjecture~\ref{conj:mainconj} manifests the equivariant $K$-theory positivity
proved (for all generalized flag varieties $G/P$) by
\cite{Anderson.Griffeth.Miller}. Let
$z_{i}:=\frac{t_i}{t_{i+1}}-1.$
Note that for $j>i$,
\begin{equation}
\label{eqn:thetransform}
1-\frac{t_i}{t_j} = -(z_i+1)\cdots(z_{j-1}+1)+1.
\end{equation}
Thus
\[{\tt sgn}(T)\cdot {\tt wt}_K(T)=(-1)^{|\nu|-|\lambda|-|\mu|}\left((-1)^{\#\mbox{\small $\star$'s in $T$}\  +  \ \#\mbox{\small edge labels in $T$}}\cdot {\tt wt}_K(T)\right).\]
Notice ${\tt wt}_{K}(T)$ is a product of
($\#\mbox{\small $\star$'s in $T$} + \#\mbox{\small edge labels in $T$}$)-many factors of the form (\ref{eqn:thetransform}) and also
$(z_i+1)\cdots(z_{j-1}+1)-1$ is manifestly positive in the variables $\{z_i\}$.
Hence Conjecture~\ref{conj:mainconj} expresses $K_{\lambda,\mu}^{\nu}\cdot (-1)^{|\nu|-|\lambda|-|\mu|}$ as a manifestly positive polynomial in the
variables $\{z_i\}$; this is the positivity of \cite{Anderson.Griffeth.Miller}, after the substitution $z_i\mapsto e^{\beta_i}-1$.

%This polynomial is inhomogeneous
%(under the standard grading that sets ${\rm deg}(z_{i})=1$). Consider
%the degree $d$ homogeneous component of $K_{\lambda,\mu}^{\nu}$.
%``Positivity'' in this setting means that any tableau
%$T$ such that ${\rm deg}({\tt wt}_K(T))=d$ has the same value of ${\tt sgn}(T)$%.
%To see this note that
%\begin{equation}
%\label{eqn:degree}
%{\rm deg}({\tt wt}_K(T))=\# \mbox{edge labels in $T$}+\#\mbox{$\star$'s
%in $T$}.
%\end{equation}
%Thus the desired conclusion holds since combining the above equation with (\ref{eqn:sign})
%gives:
%\[{\tt sgn}(T)=(-1)^{|\nu|-|\lambda|-|\mu|}(-1)^{\small {\rm deg}({\tt wt}_K(T))}.\]

We have computer verified this conjecture for all
$Gr(k,{\mathbb C}^n)$ for $n\leq 5$ as well as a number of cases
for larger $n$.

\section{Final remarks}

We are attempting to extend ideas in this paper to prove
Conjecture~\ref{conj:mainconj}. Specifically, we desire an analogue of the results of Section~2.
This would specialize to a ``semistandard'' version of the results of \cite{Thomas.Yong:V}.

 One can reformulate Theorem~\ref{thm:main} to avoid edge labels. In this version,
a bullet $\bullet$ in a box can either be replaced by a label using a classical \emph{jeu de taquin} slide or it can be replaced by a label not already present in the tableau, at
the cost of the weight associated with the box containing the $\bullet$.
\excise{Moreover,
unlike earlier rules for $C_{\lambda,\mu}^{\nu}$ (including those proved in this paper),
this rule assigns a single tableau to each monomial in the expansion of the polynomial
$C_{\lambda,\mu}^{\nu}$ into the variables $\alpha_i=t_i - t_{i+1}$. Therefore this raises the question of a finding a weight-preserving bijection between the the
set of tableau that witness $C_{\lambda,\mu}^{\nu}$ and those that witness $C_{\mu,\lambda}^{\nu}$, thus providing a direct combinatorial proof of the ``commutation symmetry'' of these coefficients. (For example, the original puzzle rule of \cite{Knutson.Tao} assigns a different number of puzzles to $C_{\lambda,\mu}^\nu$
versus $C_{\mu,\lambda}^{\nu}$. The same is true of, e.g., Theorem~\ref{thm:main}.)}

We mentioned that the equivariant cohomology of Grassmannians is controlled by multiplication of factorial
Schur polynomials. A.~Molev-B.~Sagan \cite{Molev.Sagan} introduced a generalization of this (geometrically relevant to ``triple Schubert calculus'', see \cite{Knutson.Tao}). The ideas of Section~2 also generalize to provide a \emph{jeu de taquin} rule for the
Molev-Sagan coefficients.

\section*{Acknowledgments}
We thank Allen Knutson and Mark Shimozono for helpful conversations.
HT was supported by an NSERC Discovery Grant.
AY was supported by NSF grants and by a CAS/Beckman fellowship at UIUC's Center for Advanced Study.  This project was partially completed
at the Fields Institute, and at the 2010 workshop on localization techniques in equivariant cohomology at the American Institute for Mathematics.

%\comment{Code checks}

%\comment{State the non jdt version of the rule}

\end{document}